\documentclass[10pt]{amsart}
\usepackage[margin=1.5in]{geometry}
\usepackage{xcolor}
\usepackage{xcolor}
\newtheorem{theorem}{Theorem}[section]
\newtheorem{lemma}[theorem]{Lemma}
\newtheorem{proposition}{Proposition}[section]

\theoremstyle{definition}
\newtheorem{definition}[theorem]{Definition}

\usepackage[sort,nocompress]{cite}

\usepackage{mathtools}

\theoremstyle{remark}
\newtheorem{remark}[theorem]{Remark}

\numberwithin{equation}{section}

\begin{document}
\title[BGK model for multi-component gases  near a global Maxwellian]{BGK model for multi-component gases  near a global Maxwellian}

\author{Gi-Chan Bae}
\address{Research institute of Mathematics, Seoul National University, Seoul 08826, Republic of Korea}
\email{gcbae02@snu.ac.kr}

\author{Christian Klingenberg}
\address{Department of mathematics, W\"urzburg University, Emil Fischer Str. 40, 97074 W\"urzburg, GERMANY}
\email{klingen@mathematik.uni-wuerzburg.de}

\author{Marlies Pirner}
\address{Department of mathematics, Vienna University, Oskar-Morgenstern-Platz 1, 1090 Vienna, Austria}
\email{marlies.pirner@mathematik.uni-wuerzburg.de}

\author{Seok-Bae Yun}
\address{Department of mathematics, Sungkyunkwan University, Suwon 16419, Republic of Korea }
\email{sbyun01@skku.edu}




\begin{abstract}
In this paper, we establish the existence of the unique global-in-time classical solutions to the multi-component BGK model suggested in \cite{mixmodel} when the initial data is a small perturbation of global equilibrium. 
For this, we carefully analyze the dissipative nature of the linearized multi-component relaxation operator,
and observe that the partial dissipation from the intra-species and the inter-species linearized relaxation operators are combined in a complementary manner to give rise to the desired dissipation estimate of the model
We also observe that  the convergence rate of the distribution function increases as the momentum-energy interchange rate between the different components of the gas increases.
\end{abstract}
\maketitle
\section{Introduction}
In this paper, we study the existence and the asymptotic behavior of the BGK model for multi-component gases suggested in \cite{mixmodel}:
\begin{align}\label{CCBGK}
\begin{aligned}
\partial_t F_1+v \cdot \nabla_xF_1&=n_1(\mathcal{M}_{11}-F_1)+n_2(\mathcal{M}_{12}-F_1), \cr
\partial_t F_2+v \cdot \nabla_xF_2&=n_2(\mathcal{M}_{22}-F_2)+n_1(\mathcal{M}_{21}-F_2), \cr
F_1(x,v,0)=F_{10}&(x,v), \qquad F_2(x,v,0)=F_{20}(x,v).
\end{aligned}
\end{align}
The distribution function $F_i(x,v,t)$ denotes the number density of $i$-th species particle at the phase point $(x,v) \in \mathbb{T}^3 \times \mathbb{R}^3$ at time $t\in\mathbb{R}^+$ for $i=1,2$.  The intra-species Maxwell distributions in the BGK operator $\mathcal{M}_{ii}$ are defined as 
\begin{align*}
\mathcal{M}_{ii} = \frac{n_i}{\sqrt{2\pi\frac{T_i}{m_i}}^3} \exp\left(-\frac{|v-U_i|^2}{2\frac{T_i}{m_i}}\right),\quad (i=1,2).
\end{align*}
Here $m_i$ $(i=1,2)$ denotes the mass of a molecule in the $i$-th component,
which we assume that $m_1 \geq m_2$ throughout the paper without loss of generality.
The number density $n_i$, the bulk velocity $U_i$, and the temperature $T_i$ of the $i$-th particle are defined by
\begin{align*}
n_i(x,t)&=\int_{\mathbb{R}^3}F_i(x,v,t)dv,\cr U_i(x,t)&=\frac{1}{n_i}\int_{\mathbb{R}^3}F_i(x,v,t)vdv, \cr T_i(x,t)&=\frac{1}{3n_i}\int_{\mathbb{R}^3}F_i(x,v,t)m_i|v-U_i|^2dv .
\end{align*}
The inter-species Maxwellian distributions are defined by
\begin{align*}
\mathcal{M}_{12} = \frac{n_1}{\sqrt{2\pi\frac{T_{12}}{m_1}}^3} \exp\left(-\frac{|v-U_{12}|^2}{2\frac{T_{12}}{m_1}}\right), \qquad \mathcal{M}_{21} = \frac{n_2}{\sqrt{2\pi\frac{T_{21}}{m_2}}^3} \exp\left(-\frac{|v-U_{21}|^2}{2\frac{T_{21}}{m_2}}\right),
\end{align*}
where the inter-species bulk velocities $ U_{12}, U_{21}$ and the inter-species temperatures $T_{12}, T_{21}$ are defined by
 \begin{align*}
U_{12}&=\delta U_1 + (1-\delta)U_2, \cr
U_{21}&=\frac{m_1}{m_2}(1-\delta)U_1 + \left(1-\frac{m_1}{m_2}(1-\delta)\right)U_2,
\end{align*}
and
\begin{align*}
T_{12}&= \omega T_1 + (1-\omega)T_2 + \gamma |U_2-U_1|^2, \cr
T_{21}&= (1-\omega) T_1 + \omega T_2 +\left(\frac{1}{3}m_1(1-\delta)\left(\frac{m_1}{m_2}(\delta-1)+1+\delta\right)-\gamma\right) |U_2-U_1|^2.
\end{align*}
Here, the free parameter $\delta$ and $\omega$ denote the momentum interchange rate and the temperature interchange rate, respectively.
In \eqref{CCBGK}, $n_i(\mathcal{M}_{ii}-F_i)$ $(i=1,2)$ are the intra-species relaxation operators for $i$-th gas component, while 
$n_j(\mathcal{M}_{ij}-F_i)$ $(i\neq j)$ are the inter-species relaxation operators between different components of the gas. We note that the inter-species relaxation operators describe the interchange of the macroscopic momentum and the temperature between two different species of gas.
These relaxation operators satisfy the following cancellation properties:
\begin{align*}
\begin{split}
&\int_{\mathbb{R}^3}(\mathcal{M}_{ii}-F_i) \left( 1 ,m_iv,m_i|v|^2\right)dv=0,\quad i=1,2 \cr
& \int_{\mathbb{R}^3}(\mathcal{M}_{12}-F_1) dv=0,\quad \int_{\mathbb{R}^3}(\mathcal{M}_{21}-F_2) dv=0, \cr
& \int_{\mathbb{R}^3}n_1(\mathcal{M}_{12}-F_1)m_1v dv+\int_{\mathbb{R}^3}n_2(\mathcal{M}_{21}-F_2)m_2v dv=0, \cr
& \int_{\mathbb{R}^3}n_1(\mathcal{M}_{12}-F_1)m_1|v|^2 dv+\int_{\mathbb{R}^3}n_2(\mathcal{M}_{21}-F_2)m_2|v|^2 dv=0,
\end{split}
\end{align*}
leading to the following conservation laws of the density, total momentum, and total energy:
\begin{align}\label{conservation}
\begin{split}
&\frac{d}{dt} \int_{\mathbb{T}^3 \times \mathbb{R}^3}F_1(x,v,t) dvdx = \frac{d}{dt} \int_{\mathbb{T}^3 \times \mathbb{R}^3}F_2(x,v,t) dvdx =0, \cr
&\frac{d}{dt} \int_{\mathbb{T}^3 \times \mathbb{R}^3}\left(F_1(x,v,t)m_1v + F_2(x,v,t)m_2v\right) dvdx =0, \cr
&\frac{d}{dt} \int_{\mathbb{T}^3 \times \mathbb{R}^3}\left(F_1(x,v,t)m_1|v|^2 + F_2(x,v,t)m_2|v|^2\right) dvdx =0.
\end{split}
\end{align}
To ensure the positivity of all temperatures, the free parameters $\omega$, $\delta$, and $\gamma$ are restricted to 
\begin{align*}
 \frac{ \frac{m_1}{m_2} - 1}{1+\frac{m_1}{m_2}} \leq  \delta < 1, \qquad     0 \leq \omega < 1,
\end{align*}
and
\begin{align*}
0 \leq \gamma  \leq \frac{m_1}{3}(1-\delta)\left[\left(1+\frac{m_1}{m_2} \right) \delta+1-\frac{m_1}{m_2}  \right].
\end{align*}
For more details, see \cite{mixmodel}. \newline

The main goal of this paper is to establish the global-in-time classical solution of the mixture BGK model when the initial data is close to global equilibrium. For this, we consider the following global equilibrium for each particle distribution function:
\begin{align*}
\mu_1(v)=n_{10}\frac{\sqrt{m_1}^3}{\sqrt{2\pi}^3}e^{-\frac{m_1|v|^2}{2}},\qquad \mu_2(v) = n_{20}\frac{\sqrt{m_2}^3}{\sqrt{2\pi}^3}e^{-\frac{m_2|v|^2}{2}}.
\end{align*}
We then define the perturbations $f_k$ $(k=1,2)$ by $F_k=\mu_k+\sqrt{\mu_k}f_k$ and rewrite the mixture BGK model \eqref{CCBGK} in terms of $f_k$ as
\begin{align}\label{pertf12}
\begin{split}
\partial_t f_1+v\cdot \nabla_xf_1&=L_{11}(f_1)+L_{12}(f_1,f_2)+\Gamma_{11}(f_1)+\Gamma_{12}(f_1,f_2), \cr
\partial_t f_2+v\cdot \nabla_xf_2&=L_{22}(f_2)+L_{21}(f_1,f_2)+\Gamma_{22}(f_2)+\Gamma_{21}(f_1,f_2).
\end{split}
\end{align}
On the R.H.S, $L_{11}$ and $L_{22}$ denote the linearized part of the intra-species relaxation operators:
\[L_{kk}(f_k)=n_{k0}(P_kf_k-f_k), \quad (k=1,2),
\]
where $P_k$ is the $L^2$ projection onto the linear space spanned by
\[\left\{ \sqrt{\mu_k}, v\sqrt{\mu_k}, |v|^2\sqrt{\mu_k}  \right\}.\]
The linearized operators for inter-species interactions $L_{12}$ and $L_{21}$ are given by
\begin{align*}
L_{12}(f_1,f_2) &=n_{20}(P_1f_1-f_1)\\
&\quad + n_{20} \bigg[(1-\delta) \sum_{2\leq i \leq 4}  \left(\sqrt{\frac{n_{10}}{n_{20}}}\sqrt{\frac{m_1}{m_2}}\langle f_2, e_{2i} \rangle_{L^2_v}-\langle f_1, e_{1i} \rangle_{L^2_v}\right)e_{1i} \cr
&\quad+(1-\omega) \left(\sqrt{\frac{n_{10}}{n_{20}}}\langle f_2, e_{25} \rangle_{L^2_v}-\langle f_1, e_{15} \rangle_{L^2_v}\right)e_{15} \bigg],
\end{align*}
\begin{align*}
L_{21}(f_1,f_2)  &=n_{10}(P_2f_2-f_2)\\
&\quad+ n_{10} \bigg[\frac{m_1}{m_2}(1-\delta)\sum_{2\leq i \leq 4}\left(\sqrt{\frac{n_{20}}{n_{10}}}\sqrt{\frac{m_2}{m_1}}\langle f_1, e_{1i} \rangle_{L^2_v}-\langle f_2, e_{2i} \rangle_{L^2_v}\right)e_{2i} \cr
&\quad+ (1-\omega)\left(\sqrt{\frac{n_{20}}{n_{10}}}\langle f_1, e_{15} \rangle_{L^2_v}-\langle f_2, e_{25} \rangle_{L^2_v}\right)e_{25} \bigg],
\end{align*}
for $0 \leq \delta,\omega<1$ and $\{e_{ki}\}_{1\leq i \leq 5}$ is an orthonormal basis spanned by $\left\{ \sqrt{\mu_k}, v\sqrt{\mu_k}, |v|^2\sqrt{\mu_k}  \right\}$ for $k=1,2$.
Finally, $\Gamma_{11}$, $\Gamma_{22}$, $\Gamma_{12}$, and $\Gamma_{21}$ are nonlinear perturbations.
For detailed derivation of \eqref{pertf12}, see Sec. 2.

We introduce
\[
L(f_1,f_2)=(L_{11}(f_1)+L_{12}(f_1,f_2),L_{22}(f_2)+L_{21}(f_1,f_2)),
\]
and
\[
\Gamma(f_1,f_2)=(\Gamma_{11}(f_1)+\Gamma_{12}(f_1,f_2),\Gamma_{22}(f_2)+\Gamma_{21}(f_1,f_2)),
\]
to rewrite \eqref{pertf12} in the following succinct form:
\begin{align*}
(\partial_t+v\cdot\nabla_x)(f_1,f_2)=L(f_1,f_2)+\Gamma(f_1,f_2).
\end{align*}
To state our main result, we need to set up several notations.
\begin{itemize}
	\item  The constant $C$ in the estimates will be defined generically.
	\item $\langle \cdot,\cdot\rangle_{L^2_{v}}$ and $\langle\cdot,\cdot\rangle_{L^2_{x,v}}$ denote the standard $L^2$ inner product on $\mathbb{R}^3_v$ and  $\mathbb{T}^3_x \times \mathbb{R}^3_v$, respectively.
	\begin{align*}
	\langle f,g\rangle_{L^2_{v}}=\int_{\mathbb{R}^3}f(v)g(v)dv,	\quad\langle f,g\rangle_{L^2_{x,v}}=\int_{\mathbb{T}^3\times\mathbb{R}^3}f(x,v)g(x,v)dvdx.
	\end{align*}
	\item $\|\cdot\|_{L^2_v}$ and $\|\cdot\|_{L^2_{x,v}}$ denote the standard $L^2$ norms in $\mathbb{R}^3_v$ and $\mathbb{T}^3_x \times \mathbb{R}^3_v$, respectively:
	\begin{align*}
	\|f\|_{L^2_v}\equiv \left(\int_{\mathbb{R}^3}|f(v)|^2 dv\right)^{\frac{1}{2}},	\quad\|f\|_{L^2_{x,v}}\equiv \left(\int_{\mathbb{T}^3\times\mathbb{R}^3}|f(x,v)|^2 dvdx\right)^{\frac{1}{2}}.
	\end{align*}
	\item We define an $L^2$ inner product between two vectors $(f_1,f_2)$ and $(g_1,g_2)$ as
	\begin{align*}
	&\langle(f_1,f_2),(g_1,g_2)\rangle_{L^2_{v}}=\int_{\mathbb{R}^3}f_1(v)g_1(v)+f_2(v)g_2(v)dv,	\cr
	&\langle(f_1,f_2),(g_1,g_2)\rangle_{L^2_{x,v}}=\int_{\mathbb{T}^3\times\mathbb{R}^3}f_1(x,v)g_1(x,v)+f_2(x,v)g_2(x,v)dvdx.
	\end{align*}
	\item The standard $L^2$ norm of a vector denotes
	\begin{align*}
	&\| \left(f(x,v),g(x,v)\right)  \|_{L^2_v} = \left( \int_{\mathbb{R}^3}|f(v)|^2 +|g(v)|^2 dv\right)^{\frac{1}{2}}, \cr
	&\| \left(f(x,v),g(x,v)\right)  \|_{L^2_{x,v}} = \left( \int_{\mathbb{T}^3\times\mathbb{R}^3}|f(x,v)|^2 +|g(x,v)|^2 dvdx\right)^{\frac{1}{2}}.
	\end{align*}
	\item We use the following notations for multi-indices differential operators:
	\begin{align*}
	\alpha=[\alpha_0,\alpha_1,\alpha_2,\alpha_3], \quad \beta=[\beta_1,\beta_2,\beta_3],
	\end{align*}
	and
	\begin{align*}
	\partial^{\alpha}_{\beta}=\partial_t^{\alpha_0}\partial_{x_1}^{\alpha_1}\partial_{x_2}^{\alpha_2}\partial_{x_3}^{\alpha_3}\partial_{v_1}^{\beta_1}\partial_{v_2}^{\beta_2}\partial_{v_3}^{\beta_3}.
	\end{align*}
	\item We employ the following convention for simplicity.
	\begin{align*}
	\partial^{\alpha}_{\beta}(f_1,f_2)=\big(\partial^{\alpha}_{\beta}f_1,\partial^{\alpha}_{\beta}f_2\big).
	\end{align*}
	\item We define the high-order energy norm $\mathcal{E}_{N_1,N_2}(f_1(t), f_2(t))$:
	\begin{align*}
	\mathcal{E}_{N_1,N_2}(f_1(t), f_2(t))=\sum_{\substack{|\alpha|\leq N_1,~|\beta|\leq N_2 \cr N_1+N_2=N}} \|\partial^{\alpha}_{\beta}\big(f_1(t),f_2(t)\big)\|^2_{L^2_{x,v}}.
	\end{align*}
	For notational simplicity, we use $\mathcal{E}(t)$ to denote $\mathcal{E}_{N_1,N_2}(f_1(t), f_2(t))$ when the 
	dependency on $(N_1,N_2)$ is not relevant.
\end{itemize}	
We are now ready to state our main result.
\begin{theorem} Let $N\geq 3$. We set the macroscopic quantities of the initial data to the same with that of the global equilibria:
\begin{align*}
\int_{\mathbb{T}^3 \times \mathbb{R}^3}F_{k0}(x,v) \left(\begin{array}{c} 1 \cr m_kv \cr m_k|v|^2 \end{array}\right) dvdx =  \int_{\mathbb{T}^3 \times \mathbb{R}^3}\mu_k(v) \left(\begin{array}{c} 1 \cr m_kv \cr m_k|v|^2 \end{array}\right) dvdx,
\end{align*}
for $k=1,2$.
We define $f_{k0}$ as $F_{k0}=\mu_k+ \sqrt{\mu_k}f_{k0}$. 
Then there exists $\epsilon>0$ such that if $\mathcal{E}_{N_1,N_2}(f_{10},f_{20}) < \epsilon$, then there exists a unique global-in-time classical solution of \eqref{CCBGK} satisfying   
\begin{itemize}
\item The two distribution functions are non-negative: 
\[F_k(x,v,t)=\mu_k + \sqrt{\mu_k} f_k \geq 0.\]
\item The conservation laws hold \eqref{conservation}.
\item The distribution functions converge exponentially to the global equilibrium:
\[ \mathcal{E}_{N_1,N_2}(f_1,f_2)(t) \leq Ce^{-\eta t} \mathcal{E}_{N_1,N_2}(f_{10},f_{20}).  \]
In the case of $N_2=0$, that is, if $\mathcal{E}_{N_1,0}(f_{10},f_{20}) < \epsilon$,  we have the following more detailed convergence estimate: 
\[ \mathcal{E}_{N_1,0}(f_1,f_2)(t) \leq Ce^{-\eta \min\left\{(1-\delta),(1-\omega) \right\}t} \mathcal{E}_{N_1,0}(f_{10},f_{20}).  \]
\item Let $(f_1,f_2)$ and $(\bar{f}_1,\bar{f}_2)$ be solutions corresponding to the initial data $(f_{10},f_{20})$ and $(\bar{f}_{10},\bar{f}_{20})$, respectively, then the system satisfies the following $L^2$ stability:
\[ \| (f_1 -\bar{f}_1 ,f_2 -\bar{f}_2) \|_{L^2_{x,v}}  \leq  C  \|( f_{10} -\bar{f}_{10}, f_{20} -\bar{f}_{20}) \|_{L^2_{x,v}}. \]
\end{itemize}
\end{theorem}
\begin{remark} The convergence rate in the case of $N_2=0$ shows that the higher interchange rate ($\delta$ and $\omega$ close to $0$) gives the faster convergence rates. 
\end{remark}

The most important step is the identification of the dissipation mechanism of the linearized multi-component relaxation operator. To investigate the dissipative property of $L$, we decompose the linearized inter-species relaxation operator $L_{ij} $ $(i\neq j)$ further into the mass interaction part $L^1_{ij}$ and the momentum-energy interaction part $L^2_{ij}$:
\begin{align*}
L_{12}^1(f_1)= n_{20}(P_1f_1-f_1), \quad L_{21}^1(f_2)= n_{10}(P_2f_2-f_2),
\end{align*}
and
\begin{align*}
\begin{split}
	L_{12}^2(f_1,f_2) &= n_{20} \bigg[(1-\delta) \sum_{2\leq i \leq 4}  \left(\sqrt{\frac{n_{10}}{n_{20}}}\sqrt{\frac{m_1}{m_2}}\langle f_2, e_{2i} \rangle_{L^2_v}-\langle f_1, e_{1i} \rangle_{L^2_v}\right)e_{1i} \cr
	&\quad+(1-\omega) \left(\sqrt{\frac{n_{10}}{n_{20}}}\langle f_2, e_{25} \rangle_{L^2_v}-\langle f_1, e_{15} \rangle_{L^2_v}\right)e_{15} \bigg], \cr
	L_{21}^2(f_1,f_2)  &= n_{10} \bigg[\frac{m_1}{m_2}(1-\delta)\sum_{2\leq i \leq 4}\left(\sqrt{\frac{n_{20}}{n_{10}}}\sqrt{\frac{m_2}{m_1}}\langle f_1, e_{1i} \rangle_{L^2_v}-\langle f_2, e_{2i} \rangle_{L^2_v}\right)e_{2i} \cr
	&\quad+ (1-\omega)\left(\sqrt{\frac{n_{20}}{n_{10}}}\langle f_1, e_{15} \rangle_{L^2_v}-\langle f_2, e_{25} \rangle_{L^2_v}\right)e_{25} \bigg],
\end{split}
\end{align*}
so that $L_{12}=L_{12}^1+L_{12}^2$ and $L_{21}=L_{21}^1+L_{21}^2$. We first derive from an explicit computation that the intra-species operator $L_{ii}$ and the mass interaction part of the inter-species operator $L^1_{12}$ and $L^1_{21}$ give rise to the following partial dissipative estimate:
\begin{align}\label{LtoP0}
	\begin{split}
		&\langle (L_{11}+ L_{12}^1)f_1 , f_1\rangle_{L^2_{x,v}}+\langle (L_{22}+L_{21})f_2  , f_2\rangle_{L^2_{x,v}} \cr
		&\hspace{3cm}= -(n_{10}+n_{20})\| (I-P_1,I-P_2)(f_1,f_2) \|_{L^2_{x,v}}^2.
	\end{split}
\end{align}
We note that the dissipation estimate above is too weak in that it involves 10-dimensional degeneracy,
which is 4-dimensional bigger than the 6-dimensional conservation laws in \eqref{conservation}. 
It is the additional dissipation from the momentum-energy interaction parts $L_{12}^2$, $L_{21}^2$ of the inter-species operators $L_{12}$ and $L_{21}$ that make up for the deficiency: 
\begin{align}\label{bring}
	\begin{split}
		\langle L_{12}^2,f_1 \rangle_{L^2_{x,v}}+\langle L_{21}^2,f_2 \rangle_{L^2_{x,v}} &\leq -\min\left\{(1-\delta),(1-\omega) \right\}	(n_{10}+n_{20}) \cr
		&\quad \times\left( \|(P_1,P_2)(f_1,f_2)\|_{L^2_{x,v}}^2-\|P(f_1,f_2) \|_{L^2_{x,v}}^2\right),
	\end{split}
\end{align}
where $P$ is an orthonormal  $L^2\times L^2$ projection on the space spanned by the following $6$-dimensional basis
\begin{align*}
\{(\sqrt{\mu_1},0),(0,\sqrt{\mu_2}), (m_1v\sqrt{\mu_1},m_2v\sqrt{\mu_2}),\left((m_1|v|^2-3)\sqrt{\mu_1},(m_2|v|^2-3)\sqrt{\mu_2}\right)\}.
\end{align*}
Then partial dissipation estimates \eqref{LtoP0} and \eqref{bring} complement each other to give rise to the following multi-component dissipation estimate for $L$:
\begin{align}\label{Lm}
\begin{split}
\langle L(f_1,f_2),(f_1,f_2)\rangle_{L^2_{x,v}}&\leq - (n_{10}+n_{20})\Big( \max\{\delta,\omega\}
\| (I-P_1,I-P_2)(f_1,f_2) \|_{L^2_{x,v}}^2 \cr
&\quad + \min\left\{(1-\delta),(1-\omega) \right\}\| (I-P)(f_1,f_2)\|_{L^2_{x,v}}^2 \Big).
\end{split}
\end{align}

The dissipation estimate \eqref{Lm}, together with further analysis on the degeneracy part through the standard micro-macro decomposition, provides the following full coercivity depending on the interchange rates: 
\begin{align*}
\langle L(\partial^{\alpha}(f_1,f_2)),\partial^{\alpha}(f_1,f_2)\rangle_{L^2_{x,v}} \leq - \eta \min\left\{(1-\delta),(1-\omega) \right\}\sum_{|\alpha|\leq N} \| (\partial^{\alpha}(f_1,f_2) \|_{L^2_{x,v}}^2.
\end{align*}
Due to the presence of the momentum interchange rate $\delta$ and the energy interchange rate $\omega$ between different components in the dissipation estimate, we see that the larger interchange rate (when $\delta$ and $\omega$ are close to zero) leads to the stronger dissipation, and therefore, the faster convergence to the global equilibrium: 
\begin{align*}
\sum_{|\alpha|\leq N}\|\partial^{\alpha} (f_1(t),f_2(t) )\|_{L^2_{x,v}}^2 \leq e^{-\eta \min\left\{(1-\delta),(1-\omega) \right\}t}\sum_{|\alpha|\leq N}\|\partial^{\alpha} (f_1(0), f_2(0))\|_{L^2_{x,v}}^2.
\end{align*}

\subsection{Literature review}
We start with a review of the mathematical results of the mono-species BGK model.
Perthame established the first result on global weak solutions for a general initial data in \cite{Perth}.
In \cite{PP1993}, the authors considered weighted-$L^{\infty}$ bounds to obtain the uniqueness.
Desvillettes considered the convergence to equilibrium in a weak sense \cite{Des}.
Ukai proved the existence of the stationary solution on a finite interval with inflow boundary condition in \cite{Ukai}. In \cite{ZH}, the $L^{\infty}$ work in \cite{Perth} is generalized to an weighted $L^p$ space.
Classical solutions near-global equilibrium is constructed in \cite{Bello} using the spectral analysis of Ukai \cite{Ukai spectral}, and by using the nonlinear energy method of Yan Guo \cite{Guo whole, Guo VMB, Guo VPB} in \cite{Yun1}.
The nonlinear energy method is then employed further to study several types of BGK models \cite{Yun1,Yun2,Yun3,HY,BY,Shakhov}.
Saint-Raymond considered the hydrodynamic limits of the BGK model in \cite{Saint,Saint2}.
For the numerical study of the BGK model, we refer to \cite{BCRY,MR3828279,Bennoune_2008,Crestetto_2012,Pirner4,CBRY1,CBRY2,BCRY,RSY,RY}.

Various BGK models to describe the dynamics of multi-component gases are proposed in the literature. Examples include the model of Gross and Krook \cite{gross_krook1956}, the model of Hamel \cite{hamel1965}, the model of Greene \cite{Greene}, the model of Garzo, Santos and Brey \cite{Garzo1989}, the model of Sofonea and Sekerka \cite{Sofonea2001}, the model by Andries, Aoki and Perthame  \cite{AndriesAokiPerthame2002}, the model of Brull, Pavan and Schneider \cite{Brull_2012}, the model of Klingenberg, Pirner and Puppo \cite{mixmodel}, the model of Haack, Hauck, Murillo \cite{haack}, the model of Bobylev, Bisi, Groppi, Spiga \cite{Bobylev}.
The BGK model for gas mixtures has also been extended to the ES-BGK model, polyatomic molecules, chemical reactions, or the quantum case; See for example \cite{MR3828279, Groppi, MR3960644, Pirner6, Bisi, Bisi2, Quantum,Yun,Stru}. 
For the applications of the mixture BGK models,
we refer to \cite{Puppo_2007, Jin_2010,Dimarco_2014, Bennoune_2008, Dimarco, Bernard_2015,BCRY,RSY}.
For the existence of the BGK model of gas mixtures, the mild solution was established in \cite{MR3720827}. In \cite{LiuPirner}, by constructing an entropy functional, the authors can prove exponential relaxation to equilibrium with explicit rates. The strategy is based on the entropy and spectral methods adapting Lyapunov’s direct method.

A review of the multi-species Boltzmann equation is in order.
In \cite{Guo VMB}, the author established the global existence for the mixture of a charged particle described by the Vlasov-Maxwell-Boltzmann equation.
The mild solution and uniform $L^1$ stability are obtained in \cite{HNYun}.
A mass diffusion problem of the mixture and the cross-species resonance is studied for a one-dimension case in \cite{SotiYu} based on the work in \cite{LiuYu}.
In \cite{Briant}, the author constructed the global-in-time mild solution near-global equilibrium for the mixture Boltzmann equation. 
The Vlasov–Poisson–Boltzmann equation was considered in \cite{DuanLiu} about large time asymptotic profiles when the different-species gases tend to two distinct global Maxwellians.
In \cite{GambaP}, the existence and uniqueness are constructed in spatially homogeneous settings when an initial data has upper and lower bounds for some polynomial moments.
The authors in \cite{BGPS} obtained some energy estimates.

For physical or engineering references on the studies on multi-component gases at the kinetic level,
we refer \cite{ABT,YoAo,Valo,TakaAo,SotiYu,AS,MMM,TaAo,TaAoMu,Sofonea2001}.  Some general reviews of the Boltzmann and the BGK model can be found in  \cite{B.G.K,Cercignani,DL,Chap,Cerci2,CIP,GL,V}.

This paper is organized as follows: 
In Sec. 2, we linearized the system \eqref{CCBGK} to obtain \eqref{pertf12}. 
In Sec. 3, we derive the dissipation estimate of the linearized relaxation operator. The local-in-time classical solution is constructed in Sec 4. 
In Section 5, The full coercivity of $L$ is recovered when the energy norm is sufficiently small.
Lastly, we establish the global-in-time classical solution in Sec 6.

%
%
%
%
%
%

\section{Linearization of the mixture BGK model}

\subsection{Linearization of the mixture Maxwellian}
In this part, we linearize the inter-species Maxwellian $\mathcal{M}_{12}$ and $\mathcal{M}_{21}$. We first define the macroscopic projection on $L^2_v$ and state the linearization result of the mono-species local Maxwellian $\mathcal{M}_{kk}$.
\begin{definition} We define the macroscopic projection operator $P_k$ in $L^2_v$ for $k=1,2$:
\begin{align*}
P_kf &= \frac{1}{n_{k0}} \int_{\mathbb{R}^3} f\sqrt{\mu_k} dv\sqrt{\mu_k} + \frac{m_k}{n_{k0}} \int_{\mathbb{R}^3} fv\sqrt{\mu_k} dv\cdot v\sqrt{\mu_k} \cr
&\quad + \frac{1}{6n_{k0}}\int_{\mathbb{R}^3} f(m_k|v|^2-3)\sqrt{\mu_k} dv (m_k|v|^2-3)\sqrt{\mu_k}.
\end{align*}
We denote $5$-dimensional basis as $(i=2,3,4)$
\begin{align}\label{basis}
\begin{split}
e_{k1}= \frac{1}{\sqrt{n_{k0}}}\sqrt{\mu_k},\qquad e_{ki}= \sqrt{\frac{m_k}{n_{k0}}}v_{i-1}\sqrt{\mu_k}, \qquad e_{k5} =\frac{m_k|v|^2-3}{ \sqrt{6n_{k0}} }\sqrt{\mu_k}.
\end{split}
\end{align}
\end{definition}
The $5$-dimensional basis set $\{e_{1i}\}_{i=1,\cdots,5}$ and $\{e_{2i}\}_{i=1,\cdots,5}$ construct an orthonormal basis in $L^2_v$, respectively. So, we can write 
\begin{align*}
P_1f = \sum_{1\leq i \leq 5}\langle f, e_{1i}  \rangle_{L^2_v}e_{1i}, \quad \textit{and} \quad 
P_2f = \sum_{1\leq i \leq 5}\langle f, e_{2i}  \rangle_{L^2_v}e_{2i}.
\end{align*}
\begin{lemma}\emph{\cite{Yun1}}\label{lin ii} The mono-species BGK Maxwellian $\mathcal{M}_{kk}$ is linearized as follows: 
\begin{align*} 
\mathcal{M}_{kk}(F_k) = \mu_k + \sqrt{\mu_k}P_kf_k + \sqrt{\mu_k}~\Gamma_{kk}(f_k,f_k),
\end{align*}
where the nonlinear term $\Gamma_{kk}(f_k,f_k)$ is given by
\begin{align*}
\Gamma_{kk}(f_k,f_k) &= \sum_{1 \leq i,j \leq 5} \frac{1}{\sqrt{\mu_k}}\int_0^1 \frac{P_{ij}(n_{k\theta},U_{k\theta},T_{k\theta},v-U_{k\theta},U_{k\theta})}{R_{ij}(n_{k\theta},T_{k\theta})}\mathcal{M}_{kk}(\theta)(1-\theta) d\theta \cr
&\quad \times \langle f_k,e_{ki} \rangle_{L^2_v}\langle f_k, e_{kj} \rangle_{L^2_v},
\end{align*}
for k=1,2. The function $P_{ij}(x_1,\cdots,x_5)$ denotes a generic polynomial depending on $(x_1,\cdots,x_5)$ and $R_{ij}(x,y)$ denotes a generic monomial $R_{ij}(x,y)=x^ny^m$, where $n,m\in \mathbb{N}\cup\{0\}$.
\end{lemma}
\begin{proof}
The linearization of the mono-species BGK Maxwellian $\mathcal{M}_{kk}$ is in \cite{Yun1} for the case $n_{k0}=1$ and $m_k=1$. For a general $n_{k0}$ and $m_k$, the linearization of $\mathcal{M}_{kk}$ is a special case of the linearization of $\mathcal{M}_{12}$ and $\mathcal{M}_{21}$ 
with the choice $\delta=\omega=1$ (See \eqref{lin M12} and \eqref{lin M21}, respectively).
\end{proof}

\begin{proposition}\label{linearize} The multi-species BGK Maxwellians $\mathcal{M}_{12}$ and $\mathcal{M}_{21}$ are linearized as follows:
\begin{align*}
\mathcal{M}_{12}(F)&= \mu_1 + P_1f_1\sqrt{\mu_1} 
+ (1-\delta) \sum_{2\leq i \leq 4}  \left(\sqrt{\frac{n_{10}}{n_{20}}}\sqrt{\frac{m_1}{m_2}}\langle f_2, e_{2i} \rangle_{L^2_v}-\langle f_1, e_{1i} \rangle_{L^2_v}\right)e_{1i}\sqrt{\mu_1} \cr
&\quad + (1-\omega) \left(\sqrt{\frac{n_{10}}{n_{20}}}\langle f_2, e_{25} \rangle_{L^2_v}-\langle f_1, e_{15} \rangle_{L^2_v}\right)e_{15}\sqrt{\mu_1} +\sqrt{\mu_1}\Gamma_{12}(f_1,f_2),
\end{align*}
and
\begin{align*}
\mathcal{M}_{21}(F)&= \mu_2+ P_2f_2\sqrt{\mu_2}  +  \frac{m_1}{m_2}(1-\delta)\sum_{2\leq i \leq 4}\left(\sqrt{\frac{n_{20}}{n_{10}}}\sqrt{\frac{m_2}{m_1}}\langle f_1, e_{1i} \rangle_{L^2_v}-\langle f_2, e_{2i} \rangle_{L^2_v}\right)e_{2i}\sqrt{\mu_2} \cr
&\quad + (1-\omega)\left(\sqrt{\frac{n_{20}}{n_{10}}}\langle f_1, e_{15} \rangle_{L^2_v}-\langle f_2, e_{25} \rangle_{L^2_v}\right)e_{25}\sqrt{\mu_2} +\sqrt{\mu_2}\Gamma_{21}(f_1,f_2).
\end{align*}
We give the precise definition of the nonlinear terms $\Gamma_{12}$ and $\Gamma_{21}$ in Section \ref{linMBGK}
\end{proposition}
\begin{proof}
We first define a transition of the macroscopic fields:
\begin{align}\label{transition}
n_{k\theta} = \theta n_k +(1-\theta)n_{k0}, \quad n_{k\theta}U_{k\theta} = \theta n_k U_k , \quad G_{k\theta} = \theta G_k,
\end{align}
where
\begin{align*}
G_k = \frac{3n_k T_k + m_k n_k |U_k|^2-3n_k}{\sqrt{6}},
\end{align*}
for $k=1,2$. We also denote multi-species macroscopic fields as   
\begin{align}\label{12theta}
\begin{split}
U_{12\theta}&=\delta U_{1\theta} + (1-\delta)U_{2\theta}, \cr
U_{21\theta}&=\frac{m_1}{m_2}(1-\delta)U_{1\theta} + \left(1-\frac{m_1}{m_2}(1-\delta)\right)U_{2\theta}, \cr
T_{12\theta}&= \omega T_{1\theta} + (1-\omega)T_{2\theta} + \gamma |U_{2\theta}-U_{1\theta}|^2, \cr
T_{21\theta}&= (1-\omega) T_{1\theta} + \omega T_{2\theta} +\left(\frac{1}{3}m_1(1-\delta)\left(\frac{m_1}{m_2}(\delta-1)+1+\delta\right)-\gamma\right) |U_{2\theta}-U_{1\theta}|^2 .
\end{split}
\end{align}
Then we consider the multi-species BGK Maxwellians $\mathcal{M}_{12}$ and $\mathcal{M}_{21}$, which depend on $\theta$:
\begin{align*}
\mathcal{M}_{12}(\theta) = \frac{n_{1\theta}}{\sqrt{2\pi\frac{T_{12\theta}}{m_1}}^3} \exp\left(-\frac{|v-U_{12\theta}|^2}{2\frac{T_{12\theta}}{m_1}}\right), \quad \mathcal{M}_{21}(\theta) = \frac{n_{2\theta}}{\sqrt{2\pi\frac{T_{21\theta}}{m_2}}^3} \exp\left(-\frac{|v-U_{21\theta}|^2}{2\frac{T_{21\theta}}{m_2}}\right).
\end{align*}
The definition of $n_{k\theta},U_{k\theta},T_{k\theta}$ gives 
\begin{align*}
\left(n_{k\theta},U_{k\theta},T_{k\theta} \right)|_{\theta=1} = (n_k,U_k,T_k), \quad \textit{and} \quad \left(n_{k\theta},U_{k\theta},T_{k\theta} \right)|_{\theta=0} = (n_{k0},0,1),
\end{align*}
so we have
\begin{align*}
\mathcal{M}_{12}(1) = \mathcal{M}_{12}, \quad  \mathcal{M}_{12}(0) = \mu_1,
\end{align*}
and
\begin{align*}
\mathcal{M}_{21}(1) = \mathcal{M}_{21}, \quad  \mathcal{M}_{21}(0) = \mu_2,
\end{align*}
where we used $U_{120} = U_{210}=0$ and $T_{120} = T_{210}=1$. 
We apply the Taylor expansion to $\mathcal{M}_{12}(\theta)$ and $\mathcal{M}_{21}(\theta)$:
\begin{align*}
\mathcal{M}_{12}(1)=\mu_1+\mathcal{M}_{12}'(0)+\int_0^1\mathcal{M}_{12}''(\theta)(1-\theta)d\theta,
\end{align*}
and
\begin{align*}
\mathcal{M}_{21}(1)=\mu_2+\mathcal{M}_{21}'(0)+\int_0^1\mathcal{M}_{21}''(\theta)(1-\theta)d\theta.
\end{align*}
By the chain rule, we compute the linear term $\mathcal{M}_{ij}'(0)$:
\begin{multline}\label{M'(0)}
\mathcal{M}_{ij}'(0)=\left(\frac{d (n_{1\theta}, n_{1\theta}U_{1\theta}, G_{1\theta},n_{2\theta}, n_{2\theta}U_{2\theta}, G_{2\theta})}{d \theta}\right)^{T}   \cr
\times \left( \frac{\partial(n_{1\theta}, n_{1\theta}U_{1\theta}, G_{1\theta},n_{2\theta}, n_{2\theta}U_{2\theta}, G_{2\theta})} {\partial(n_{1\theta},U_{1\theta},T_{1\theta},n_{2\theta},U_{2\theta},T_{2\theta})} \right)^{-1} \left(\nabla_{(n_{1\theta},U_{1\theta},T_{1\theta},n_{2\theta},U_{2\theta},T_{2\theta})}\mathcal{M}_{ij}(\theta)\right)\bigg|_{\theta=0},
\end{multline}  
for $(i,j)=(1,2)$ or $(2,1)$.
Although $\mathcal{M}_{12}$ does not depend on $n_2$, we use the above form for the convenience of the calculation. In this proposition, we focus on the linear term $\mathcal{M}_{12}'(0)$ and $\mathcal{M}_{21}'(0)$.
The exact form of the nonlinear terms will be presented in Section \ref{linMBGK}.
The remaining proof proceeds by stating some auxiliary lemmas below.
\end{proof}

\begin{lemma}\emph{\cite{Yun1}}\label{Jaco} Let us define
\[G = \frac{3n T + m n |U|^2-3n}{\sqrt{6}}. \]
Then we have 
\begin{align*}
J = \frac{\partial(n,n U,G)} {\partial(n,U,T)} = \left[ {\begin{array}{cccccc}
1 & 0 & 0 & 0 & 0  \\ 
U_1 & n & 0 & 0 & 0 \\
U_2 & 0 & n & 0 & 0 \\
U_3 & 0 & 0 & n & 0 \\
\frac{3T+m|U|^2-3}{\sqrt{6}} & \frac{2n U_1m}{\sqrt{6}} & \frac{2n U_2m}{\sqrt{6}} & \frac{2n U_3m}{\sqrt{6}} & \frac{3n }{\sqrt{6}}
\end{array} } \right],
\end{align*}
and
\begin{align*}
J^{-1} = \left(\frac{\partial(n,n U,G)} {\partial(n,U,T)}\right)^{-1} = \left[ {\begin{array}{cccccc}
1 & 0 & 0 & 0 & 0  \\ 
-\frac{U_1}{n} & \frac{1}{n} & 0 & 0 & 0 \\
-\frac{U_2}{n} & 0 & \frac{1}{n} & 0 & 0 \\
-\frac{U_3}{n} & 0 & 0 & \frac{1}{n} & 0 \\
\frac{m|U|^2-3T+3}{3n} & -\frac{2m}{3}\frac{U_1}{n} & -\frac{2m}{3}\frac{U_2}{n} & -\frac{2m}{3}\frac{U_3}{n} & \sqrt{\frac{2}{3}}\frac{1}{n}
\end{array} } \right].
\end{align*}
\end{lemma}
\begin{proof}
In the case of $m_i=1$, it is proved in \cite{Yun1}, and by the same explicit calculation, we can extend the result for general $m_i$. We omit it.
\end{proof}
\subsubsection{Linearization of $\mathcal{M}_{12}$} We first consider the calculation of $\mathcal{M}'_{12}(0)$ in \eqref{M'(0)}.
\begin{lemma}\label{M_12 diff} We have
\begin{align*}
&(1) ~ \frac{\partial \mathcal{M}_{12}(\theta)}{\partial n_{1\theta}}\bigg|_{\theta=0} =\frac{1}{n_{10}}\mu_1,
&(2) ~ \frac{\partial \mathcal{M}_{12}(\theta)}{\partial U_{1\theta}}\bigg|_{\theta=0} &= \delta m_1 v \mu_1, \cr 
&(3) ~ \frac{\partial \mathcal{M}_{12}(\theta)}{\partial T_{1\theta}}\bigg|_{\theta=0} = \omega \frac{m_1|v|^2-3}{2}\mu_1,
&(4) ~\frac{\partial \mathcal{M}_{12}(\theta)}{\partial U_{2\theta}}\bigg|_{\theta=0} &= (1-\delta)m_1 v\mu_1, \cr
&(5) ~ \frac{\partial \mathcal{M}_{12}(\theta)}{\partial T_{2\theta}}\bigg|_{\theta=0} = (1-\omega)\frac{m_1|v|^2-3}{2}\mu_1. &
\end{align*}
\end{lemma}
\begin{proof}
For readability, we ignore the dependence on $\theta$.
\newline
(1) By an explicit computation, we have 
\begin{align*}
\frac{\partial \mathcal{M}_{12}}{\partial n_1} = \frac{1}{n_1}\mathcal{M}_{12}.
\end{align*}
(2) Note that both $U_{12}$ and $T_{12}$ depend on $U_1$. So that, the chain rule gives 
\begin{align*}
\frac{\partial \mathcal{M}_{12}}{\partial U_1} &= \frac{\partial U_{12}}{\partial U_1}\frac{\partial \mathcal{M}_{12}}{\partial U_{12}} + \frac{\partial T_{12}}{\partial U_1}\frac{\partial \mathcal{M}_{12}}{\partial T_{12}} \cr
&= \delta m_1\frac{v-U_{12}}{T_{12}}\mathcal{M}_{12} -2\gamma(U_2-U_1)\left(-\frac{3}{2}\frac{1}{T_{12}}+\frac{m_1|v-U_{12}|^2}{2T_{12}^2}\right)\mathcal{M}_{12}.
\end{align*}
(3) An explicit calculation gives
\begin{align*}
\frac{\partial \mathcal{M}_{12}}{\partial T_1} = \frac{\partial T_{12}}{\partial T_1}\frac{\partial \mathcal{M}_{12}}{\partial T_{12}} = \omega \left(-\frac{3}{2}\frac{1}{T_{12}}+\frac{m_1|v-U_{12}|^2}{2T_{12}^2}\right)\mathcal{M}_{12}.
\end{align*}
(4) Similar to case (2), both $U_{12}$ and $T_{12}$ depend on $U_2$.
\begin{align*}
\frac{\partial \mathcal{M}_{12}}{\partial U_2} &= \frac{\partial U_{12}}{\partial U_2}\frac{\partial \mathcal{M}_{12}}{\partial U_{12}} + \frac{\partial T_{12}}{\partial U_2}\frac{\partial \mathcal{M}_{12}}{\partial T_{12}} \cr
&= (1-\delta)m_1\frac{v-U_{12}}{T_{12}}\mathcal{M}_{12} +2\gamma(U_2-U_1)\left(-\frac{3}{2}\frac{1}{T_{12}}+\frac{m_1|v-U_{12}|^2}{2T_{12}^2}\right)\mathcal{M}_{12}.
\end{align*}
(5) By an explicit computation, we have
\begin{align*}
\frac{\partial \mathcal{M}_{12}}{\partial T_2} = \frac{\partial T_{12}}{\partial T_2}\frac{\partial \mathcal{M}_{12}}{\partial T_{12}} = (1-\omega) \left(-\frac{3}{2}\frac{1}{T_{12}}+\frac{m_1|v-U_{12}|^2}{2T_{12}^2}\right)\mathcal{M}_{12}.
\end{align*}
Substituting
\begin{align}\label{1020}
(n_{1\theta},U_{1\theta},T_{1\theta},U_{2\theta},T_{2\theta})\big|_{\theta=0}&= (n_{10},U_{10},T_{10},U_{20},T_{20}) = (n_{10},0,1,0,1),
\end{align}
and
\begin{align}\label{12210}
U_{12\theta}|_{\theta=0}=U_{21\theta}|_{\theta=0} = 0, \quad \quad 
T_{12\theta}|_{\theta=0}=T_{21\theta}|_{\theta=0} = 1,
\end{align}
on the above computations, we get the desired result. \newline
\end{proof}
Now we proceed with the proof of Proposition \ref{linearize} for $\mathcal{M}_{12}(F)$. By the definition of the transition of the macroscopic fields \eqref{transition} and the definition of the basis \eqref{basis}, we have 
\begin{align}\label{lin1}
\begin{split}
\frac{d (n_{k\theta}, n_{k\theta}U_{k\theta}, G_{k\theta})}{d \theta}
&= \left( \int_{\mathbb{R}^3} f_k\sqrt{\mu_k} dv, \int_{\mathbb{R}^3} f_kv\sqrt{\mu_k}dv, \int _{\mathbb{R}^3}f_k\frac{m_k|v|^2-3}{\sqrt{6}}\sqrt{\mu_k} dv \right) \cr
&= \left( \langle f_k,e_{k1} \rangle_{L^2_v} , \langle f_k,e_{k2} \rangle_{L^2_v}, \langle f_k,e_{k3} \rangle_{L^2_v}, \langle f_k,e_{k4} \rangle_{L^2_v} , \langle f_k,e_{k5} \rangle_{L^2_v} \right),
\end{split}
\end{align}
for $k=1,2$. For notational brevity, we define
\[J_{k\theta} = \frac{\partial(n_{k\theta},n_{k\theta}U_{k\theta},G_{k\theta})} {\partial(n_{k\theta},U_{k\theta},T_{k\theta})}.\]
Then applying Lemma \ref{Jaco} gives 
\begin{align*}
J^{-1}_{k\theta}\big|_{\theta=0} = diag\left(1,\frac{1}{n_{k0}},\frac{1}{n_{k0}},\frac{1}{n_{k0}},\sqrt{\frac{2}{3}}\frac{1}{n_{k0}}\right),
\end{align*}
and
\begin{align}\label{lin2}
\begin{split}
\left( \frac{\partial(n_{1\theta}, n_{1\theta}U_{1\theta}, G_{1\theta},n_{2\theta}, n_{2\theta}U_{2\theta}, G_{2\theta})} {\partial(n_{1\theta},U_{1\theta},T_{1\theta},n_{2\theta},U_{2\theta},T_{2\theta})} \right)^{-1}\bigg|_{\theta=0} = \left[ {\begin{array}{cccccc}
J^{-1}_{1\theta}\big|_{\theta=0} & 0 \\
			0 & J^{-1}_{2\theta}\big|_{\theta=0}
	\end{array} } \right] ,
\end{split}
\end{align}
where we used 
\begin{align}\label{block inv}
\begin{split}
\left[ {\begin{array}{cccccc}
		J_1 & 0 \\
		0 & J_2
\end{array} } \right]^{-1}= 
\left[ {\begin{array}{cccccc}
			J^{-1}_1 & 0 \\
			0 & J^{-1}_2
	\end{array} } \right].
\end{split}
\end{align}
We substitute \eqref{lin1}, \eqref{lin2} and Lemma \ref{M_12 diff} into \eqref{M'(0)} to obtain
\begin{align*}
\mathcal{M}_{12}'(0)&=\frac{\mu_1}{n_{10}}\int_{\mathbb{R}^3} f_1\sqrt{\mu_1} dv + \frac{\delta m_1 v \mu_1}{n_{10}}\int_{\mathbb{R}^3} f_1v\sqrt{\mu_1}dv  \cr
&\quad+ \omega \frac{m_1|v|^2-3}{2}\mu_1\sqrt{\frac{2}{3}}\frac{1}{n_{10}} \int _{\mathbb{R}^3}f_1\frac{m_1|v|^2-3}{\sqrt{6}}\sqrt{\mu_1} dv \cr
&\quad+ \frac{(1-\delta)m_1v\mu_1}{n_{20}}\int_{\mathbb{R}^3} f_2v\sqrt{\mu_2}dv \cr
&\quad+ (1-\omega)\frac{m_1|v|^2-3}{2}\mu_1\sqrt{\frac{2}{3}}\frac{1}{n_{20}}\int _{\mathbb{R}^3}f_2\frac{m_2|v|^2-3}{\sqrt{6}}\sqrt{\mu_2} dv.
\end{align*}
Using the definition of the basis in \eqref{basis}, we simplify it as follows:
\begin{multline}\label{lin M12}
\mathcal{M}_{12}'(0)=\langle f_1, e_{11} \rangle_{L^2_v}e_{11}\sqrt{\mu_1} + \delta  \sum_{2\leq i \leq 4}\langle f_1, e_{1i} \rangle_{L^2_v}e_{1i}\sqrt{\mu_1} + \omega  \langle f_1, e_{15} \rangle_{L^2_v}e_{15}\sqrt{\mu_1} \cr
+  (1-\delta)\sqrt{\frac{n_{10}}{n_{20}}}\sqrt{\frac{m_1}{m_2}}\sum_{2\leq i \leq 4}\langle f_2, e_{2i} \rangle_{L^2_v}e_{1i}\sqrt{\mu_1} + (1-\omega) \sqrt{\frac{n_{10}}{n_{20}}}\langle f_2, e_{25} \rangle_{L^2_v}e_{15}\sqrt{\mu_1}.
\end{multline}
Adding and subtracting the following term
\begin{align*}
(1-\delta) \sum_{2\leq i \leq 4}\langle f_1, e_{1i} \rangle_{L^2_v}e_{1i}\sqrt{\mu_1} + (1-\omega) \langle f_1, e_{15} \rangle_{L^2_v}e_{15}\sqrt{\mu_1},
\end{align*}
gives 
\begin{align*}
\mathcal{M}_{12}'(0)&= P_1f_1\sqrt{\mu_1} 
+ (1-\delta) \sum_{2\leq i \leq 4}  \left(\sqrt{\frac{n_{10}}{n_{20}}}\sqrt{\frac{m_1}{m_2}}\langle f_2, e_{2i} \rangle_{L^2_v}-\langle f_1, e_{1i} \rangle_{L^2_v}\right)e_{1i}\sqrt{\mu_1} \cr
&\quad+ (1-\omega) \left(\sqrt{\frac{n_{10}}{n_{20}}}\langle f_2, e_{25} \rangle_{L^2_v}-\langle f_1, e_{15} \rangle_{L^2_v}\right)e_{15}\sqrt{\mu_1}.
\end{align*}
This completes the proof for the linearization of $\mathcal{M}_{12}$.

\subsubsection{Linearization of $\mathcal{M}_{21}$} Now we consider the calculation of $\mathcal{M}_{21}$ in \eqref{M'(0)}.

\begin{lemma}\label{M_21 diff} We have 
\begin{align*}
&(1) ~ \frac{\partial \mathcal{M}_{21\theta}}{\partial n_{2\theta}}\bigg|_{\theta=0} =\frac{1}{n_{20}}\mu_2,  
&(2)~ \frac{\partial \mathcal{M}_{21\theta}}{\partial U_{2\theta}}\bigg|_{\theta=0} &=\left(1-\frac{m_1}{m_2}(1-\delta)\right) m_2 v \mu_2, \cr
&(3)~ \frac{\partial \mathcal{M}_{21\theta}}{\partial T_{2\theta}}\bigg|_{\theta=0} = \omega \frac{m_2|v|^2-3}{2}\mu_2 , 
&(4)~ \frac{\partial \mathcal{M}_{21\theta}}{\partial U_{1\theta}}\bigg|_{\theta=0}&= \frac{m_1}{m_2}(1-\delta)m_2 v\mu_2, \cr
&(5) ~ \frac{\partial \mathcal{M}_{21\theta}}{\partial T_{1\theta}}\bigg|_{\theta=0} = (1-\omega)\frac{m_2|v|^2-3}{2}\mu_2. &
\end{align*}
\end{lemma}
\begin{proof}
(1) By an explicit computation, we have 
\begin{align*}
\frac{\partial \mathcal{M}_{21}}{\partial n_2} = \frac{1}{n_2}\mathcal{M}_{21}.
\end{align*}
(2) Note that $U_{21}$ and $T_{21}$ depend on $U_2$. The chain rule gives 
\begin{align*}
\frac{\partial \mathcal{M}_{21}}{\partial U_2} &= \frac{\partial U_{21}}{\partial U_2}\frac{\partial \mathcal{M}_{21}}{\partial U_{21}} + \frac{\partial T_{21}}{\partial U_2}\frac{\partial \mathcal{M}_{21}}{\partial T_{21}}.
\end{align*}
So we differentiate
\begin{align*}
\frac{\partial U_{21}}{\partial U_2}\frac{\partial \mathcal{M}_{21}}{\partial U_{21}}&= (1-\frac{m_1}{m_2}(1-\delta)) m_2\frac{v-U_{21}}{T_{21}}\mathcal{M}_{21},
\end{align*}
and
\begin{align*}
\frac{\partial T_{21}}{\partial U_2}\frac{\partial \mathcal{M}_{21}}{\partial T_{21}} 
&=2\left(\frac{1}{3}m_1(1-\delta)\left(\frac{m_1}{m_2}(\delta-1)+1+\delta\right)-\gamma\right)\cr
&\quad \times (U_2-U_1)\left(-\frac{3}{2}\frac{1}{T_{21}}+\frac{m_2|v-U_{21}|^2}{2T_{21}^2}\right)\mathcal{M}_{21}.
\end{align*}
(3) We have
\begin{align*}
\frac{\partial \mathcal{M}_{21}}{\partial T_2} = \frac{\partial T_{21}}{\partial T_2}\frac{\partial \mathcal{M}_{21}}{\partial T_{21}} = \omega \left(-\frac{3}{2}\frac{1}{T_{21}}+\frac{m_2|v-U_{21}|^2}{2T_{21}^2}\right)\mathcal{M}_{21}.
\end{align*}
(4) Since both $U_{21}$ and $T_{21}$ depend on $U_1$, 
\begin{align*}
\frac{\partial \mathcal{M}_{21}}{\partial U_1} &= \frac{\partial U_{21}}{\partial U_1}\frac{\partial \mathcal{M}_{21}}{\partial U_{21}} + \frac{\partial T_{21}}{\partial U_1}\frac{\partial \mathcal{M}_{21}}{\partial T_{21}},
\end{align*}
we compute
\begin{align*}
\frac{\partial U_{21}}{\partial U_1}\frac{\partial \mathcal{M}_{21}}{\partial U_{21}} &= \frac{m_1}{m_2}(1-\delta)m_2\frac{v-U_{21}}{T_{21}}\mathcal{M}_{21}, 
\end{align*}
and
\begin{align*}
\frac{\partial T_{21}}{\partial U_1}\frac{\partial \mathcal{M}_{21}}{\partial T_{21}} &= -2\left(\frac{1}{3}m_1(1-\delta)\left(\frac{m_1}{m_2}(\delta-1)+1+\delta\right)-\gamma\right)\cr
&\quad \times (U_2-U_1)\left(-\frac{3}{2}\frac{1}{T_{21}}+\frac{m_2|v-U_{21}|^2}{2T_{21}^2}\right)\mathcal{M}_{21}.
\end{align*}
(5) We have 
\begin{align*}
\frac{\partial \mathcal{M}_{21}}{\partial T_1} = \frac{\partial T_{21}}{\partial T_1}\frac{\partial \mathcal{M}_{21}}{\partial T_{21}} = (1-\omega) \left(-\frac{3}{2}\frac{1}{T_{21}}+\frac{m_2|v-U_{21}|^2}{2T_{21}^2}\right)\mathcal{M}_{21}.
\end{align*}
Similar to Lemma \ref{M_12 diff}, substituting \eqref{1020} and \eqref{12210} on the above calculations gives desired results.
\end{proof}

Substituting \eqref{lin1}, \eqref{lin2}, and Lemma \ref{M_21 diff} into \eqref{M'(0)} yields 
\begin{align*}
\mathcal{M}_{21}'(0)&=\frac{\mu_2}{n_{20}}\int_{\mathbb{R}^3} f_2\sqrt{\mu_2} dv + \frac{\left(1-\frac{m_1}{m_2}(1-\delta)\right) m_2 v \mu_2}{n_{20}}\int_{\mathbb{R}^3} f_2v\sqrt{\mu_2}dv \cr
&\quad+ \omega \frac{m_2|v|^2-3}{2}\mu_2 \sqrt{\frac{2}{3}}\frac{1}{n_{20}} \int _{\mathbb{R}^3}f_2\frac{m_2|v|^2-3}{\sqrt{6}}\sqrt{\mu_2} dv \cr
&\quad+ \frac{\frac{m_1}{m_2}(1-\delta)m_2 v\mu_2}{n_{10}}\int_{\mathbb{R}^3} f_1v\sqrt{\mu_1}dv \cr
&\quad+ (1-\omega)\frac{m_2|v|^2-3}{2}\mu_2\sqrt{\frac{2}{3}}\frac{1}{n_{10}}\int _{\mathbb{R}^3}f_1\frac{m_1|v|^2-3}{\sqrt{6}}\sqrt{\mu_1} dv .
\end{align*}
Using the notation of the basis in \eqref{basis}, it is equal to 
\begin{align}\label{lin M21}
\begin{split}
\mathcal{M}_{21}'(0)&=\langle f_2, e_{21} \rangle_{L^2_v}e_{21}\sqrt{\mu_2} \cr
&\quad+ \left(1-\frac{m_1}{m_2}(1-\delta)\right)  \sum_{2\leq i \leq 4}\langle f_2, e_{2i} \rangle_{L^2_v}e_{2i}\sqrt{\mu_2} + \omega  \langle f_2, e_{25} \rangle_{L^2_v}e_{25}\sqrt{\mu_2} \cr
&\quad+  \frac{m_1}{m_2}(1-\delta)\sqrt{\frac{n_{20}}{n_{10}}}\sqrt{\frac{m_2}{m_1}}\sum_{2\leq i \leq 4}\langle f_1, e_{1i} \rangle_{L^2_v}e_{2i}\sqrt{\mu_2}\cr
&\quad+ (1-\omega)\sqrt{\frac{n_{20}}{n_{10}}} \langle f_1, e_{15} \rangle_{L^2_v}e_{25}\sqrt{\mu_2}.
\end{split}
\end{align}
Adding and subtracting the following term
\begin{align*}
\frac{m_1}{m_2}(1-\delta)  \sum_{2\leq i \leq 4}\langle f_2, e_{2i} \rangle_{L^2_v}e_{2i}\sqrt{\mu_2} + (1-\omega)  \langle f_2, e_{25} \rangle_{L^2_v}e_{25}\sqrt{\mu_2},
\end{align*}
gives
\begin{align*}
\mathcal{M}_{21}'(0)&= P_2f_2\sqrt{\mu_2}  \cr
&\quad+  \frac{m_1}{m_2}(1-\delta)\sum_{2\leq i \leq 4}\left(\sqrt{\frac{n_{20}}{n_{10}}}\sqrt{\frac{m_2}{m_1}}\langle f_1, e_{1i} \rangle_{L^2_v}-\langle f_2, e_{2i} \rangle_{L^2_v}\right)e_{2i}\sqrt{\mu_2} \cr
&\quad+ (1-\omega)\left(\sqrt{\frac{n_{20}}{n_{10}}}\langle f_1, e_{15} \rangle_{L^2_v}-\langle f_2, e_{25} \rangle_{L^2_v}\right)e_{25}\sqrt{\mu_2}.
\end{align*}
This completes the proof for the linearization of $\mathcal{M}_{21}$.

\subsection{Linearization of the mixture BGK model}\label{linMBGK} In this part, we linearize the mixture BGK model \eqref{CCBGK}. Applying the linearization of the BGK Maxwellian Lemma \ref{lin ii} and Proposition \ref{linearize}, we substitute $F_1=\mu_1+\sqrt{\mu_1}f_1$ on $(\ref{CCBGK})_1$ and divide it by $\sqrt{\mu_1}$ to have
\begin{align*}
\partial_t f_1+v\cdot \nabla_xf_1&=n_1(P_1f_1-f_1 +\frac{1}{\sqrt{\mu_1}}\int_0^1\mathcal{M}_{11}''(\theta)(1-\theta)d\theta)\cr
&\quad+n_2(P_1f_1-f_1 +\frac{1}{\sqrt{\mu_1}}\int_0^1\mathcal{M}_{12}''(\theta)(1-\theta)d\theta) \cr
&\quad+n_2 \bigg[(1-\delta) \sum_{2\leq i \leq 4}  \left(\sqrt{\frac{n_{10}}{n_{20}}}\sqrt{\frac{m_1}{m_2}}\langle f_2, e_{2i} \rangle_{L^2_v}-\langle f_1, e_{1i} \rangle_{L^2_v}\right)e_{1i} \cr
&\quad+(1-\omega) \left(\sqrt{\frac{n_{10}}{n_{20}}}\langle f_2, e_{25} \rangle_{L^2_v}-\langle f_1, e_{15} \rangle_{L^2_v}\right)e_{15} \bigg].
\end{align*}
Splitting $n_k$ by $n_k=(n_k-n_{k0})+n_{k0}$,
\begin{align}\label{rho decomp}
\begin{split}
n_k &= n_k-n_{k0} +n_{k0} = \int_{\mathbb{R}^3}f_k\sqrt{\mu_k} dv +n_{k0} =\sqrt{n_{k0}}\langle f_k, e_{k1} \rangle_{L^2_v}+n_{k0},
\end{split}
\end{align}
we can have the following linearized equation:
\begin{align}\label{pertf1}
\partial_t f_1+v\cdot \nabla_xf_1&=L_{11}(f_1)+L_{12}(f_1,f_2)+\Gamma_{11}(f_1)+\Gamma_{12}(f_1,f_2),
\end{align}
where $L_{11}(f_1)=n_{10}(P_1f_1-f_1)$. The linear term $L_{12}$ is decomposed as $L_{12}=L_{12}^1+L_{12}^2$ with $L_{12}^1= n_{20}(P_1f_1-f_1)$. And $L_{12}^2$ denotes the linear term describing the interchange of momentum and temperature of each species as follows:
\begin{align}\label{def_AL1}
\begin{split}
L_{12}^2(f_1,f_2) &= n_{20} \bigg[(1-\delta) \sum_{2\leq i \leq 4}  \left(\sqrt{\frac{n_{10}}{n_{20}}}\sqrt{\frac{m_1}{m_2}}\langle f_2, e_{2i} \rangle_{L^2_v}-\langle f_1, e_{1i} \rangle_{L^2_v}\right)e_{1i} \cr
&\quad+(1-\omega) \left(\sqrt{\frac{n_{10}}{n_{20}}}\langle f_2, e_{25} \rangle_{L^2_v}-\langle f_1, e_{15} \rangle_{L^2_v}\right)e_{15} \bigg].
\end{split}
\end{align}
The nonlinear terms $\Gamma_{11}$ and $\Gamma_{12}$ denote
\begin{align*}
\Gamma_{11}(f_1) &= (n_1-n_{10})(P_1f_1-f_1) + n_1 \frac{1}{\sqrt{\mu_1}} \int_0^1\mathcal{M}_{11}''(\theta)(1-\theta)d\theta, \cr
\Gamma_{12}(f_1,f_2) &= (n_2-n_{20})(P_1f_1-f_1)+n_2 \frac{1}{\sqrt{\mu_1}} \int_0^1\mathcal{M}_{12}''(\theta)(1-\theta)d\theta \cr
&\quad+ (n_2-n_{20})\bigg[(1-\delta) \sum_{2\leq i \leq 4}  \left(\sqrt{\frac{n_{10}}{n_{20}}}\sqrt{\frac{m_1}{m_2}}\langle f_2, e_{2i} \rangle_{L^2_v}-\langle f_1, e_{1i} \rangle_{L^2_v}\right)e_{1i} \cr
&\quad+(1-\omega) \left(\sqrt{\frac{n_{10}}{n_{20}}}\langle f_2, e_{25} \rangle_{L^2_v}-\langle f_1, e_{15} \rangle_{L^2_v}\right)e_{15} \bigg].
\end{align*}
Similarly, we substitute $F_2=\mu_2+\sqrt{\mu_2}f_2$ on $(\ref{CCBGK})_2$ and divide it by $\sqrt{\mu_2}$ to have
\begin{align*}
\partial_t f_2+v\cdot \nabla_xf_2&=n_2(P_2f_2-f_2+\frac{1}{\sqrt{\mu_2}}\int_0^1\mathcal{M}_{22}''(\theta)(1-\theta)d\theta)\cr
&\quad +n_1(P_2f_2-f_2+\frac{1}{\sqrt{\mu_2}}\int_0^1\mathcal{M}_{21}''(\theta)(1-\theta)d\theta) \cr
&\quad+ n_1 \bigg[\frac{m_1}{m_2}(1-\delta)\sum_{2\leq i \leq 4}\left(\sqrt{\frac{n_{20}}{n_{10}}}\sqrt{\frac{m_2}{m_1}}\langle f_1, e_{1i} \rangle_{L^2_v}-\langle f_2, e_{2i} \rangle_{L^2_v}\right)e_{2i} \cr
&\quad+ (1-\omega)\left(\sqrt{\frac{n_{20}}{n_{10}}}\langle f_1, e_{15} \rangle_{L^2_v}-\langle f_2, e_{25} \rangle_{L^2_v}\right)e_{25} \bigg],
\end{align*}
which yields 
\begin{align}\label{pertf2}
\partial_t f_2+v\cdot \nabla_xf_2&=L_{22}(f_2)+L_{21}^2(f_1,f_2)+\Gamma_{22}(f_2)+\Gamma_{21}(f_1,f_2),
\end{align}
where $L_{22}(f_2)=n_{20}(P_2f_2-f_2)$. The linear term $L_{21}$ also decomposed as $L_{21}=L_{21}^1+L_{21}^2$ with $L_{21}^1=n_{10}(P_2f_2-f_2)$. And $L_{21}^2$ denotes the interchange of the momentum and temperature between other species.
\begin{align*}
L_{21}^2(f_1,f_2)  &=n_{10} \bigg[\frac{m_1}{m_2}(1-\delta)\sum_{2\leq i \leq 4}\left(\sqrt{\frac{n_{20}}{n_{10}}}\sqrt{\frac{m_2}{m_1}}\langle f_1, e_{1i} \rangle_{L^2_v}-\langle f_2, e_{2i} \rangle_{L^2_v}\right)e_{2i} \cr
&\quad+ (1-\omega)\left(\sqrt{\frac{n_{20}}{n_{10}}}\langle f_1, e_{15} \rangle_{L^2_v}-\langle f_2, e_{25} \rangle_{L^2_v}\right)e_{25} \bigg].
\end{align*}
The nonlinear terms $\Gamma_{22}$ and $\Gamma_{21}$ denote 
\begin{align*}
\Gamma_{22}(f_2) &= (n_2-n_{20})(P_2f_2-f_2) + n_2 \frac{1}{\sqrt{\mu_2}} \int_0^1\mathcal{M}_{22}''(\theta)(1-\theta)d\theta, \cr
\Gamma_{21}(f_1,f_2) &= (n_1-n_{10})(P_2f_2-f_2)+n_1 \frac{1}{\sqrt{\mu_2}} \int_0^1\mathcal{M}_{21}''(\theta)(1-\theta)d\theta \cr
&\quad+ (n_1-n_{10}) \bigg[\frac{m_1}{m_2}(1-\delta)\sum_{2\leq i \leq 4}\left(\sqrt{\frac{n_{20}}{n_{10}}}\sqrt{\frac{m_2}{m_1}}\langle f_1, e_{1i} \rangle_{L^2_v}-\langle f_2, e_{2i} \rangle_{L^2_v}\right)e_{2i} \cr
&\quad+ (1-\omega)\left(\sqrt{\frac{n_{20}}{n_{10}}}\langle f_1, e_{15} \rangle_{L^2_v}-\langle f_2, e_{25} \rangle_{L^2_v}\right)e_{25} \bigg].
\end{align*}
Overall, we can write the linearized mixture BGK model \eqref{CCBGK} as 
\begin{align}\label{linf}
\begin{split}
\partial_t f_1+v\cdot \nabla_xf_1&=L_{11}(f_1)+L_{12}(f_1,f_2)+\Gamma_{11}(f_1)+\Gamma_{12}(f_1,f_2), \cr
\partial_t f_2+v\cdot \nabla_xf_2&=L_{22}(f_2)+L_{21}(f_1,f_2)+\Gamma_{22}(f_2)+\Gamma_{21}(f_1,f_2), \cr
f_1(x,v,0)=f_{10}&(x,v), \qquad f_2(x,v,0)=f_{20}(x,v).
\end{split}
\end{align}
where $f_{10} =(F_{10}-\mu_1)/\sqrt{\mu_1}$, and $f_{20} = (F_{20}-\mu_2)/\sqrt{\mu_2}$. The linearized mixture BGK model \eqref{linf} satisfies the following conservation laws.
\begin{align}\label{conservf}
\begin{split}
&\int_{\mathbb{T}^3 \times \mathbb{R}^3}\sqrt{\mu_1}f_1(x,v,t) dvdx = \int_{\mathbb{T}^3 \times \mathbb{R}^3}\sqrt{\mu_2}f_2(x,v,t) dvdx =0, \cr
&\int_{\mathbb{T}^3 \times \mathbb{R}^3}\left(\sqrt{\mu_1}f_1(x,v,t)m_1v + \sqrt{\mu_2}f_2(x,v,t)m_2v\right) dvdx =0, \cr
&\int_{\mathbb{T}^3 \times \mathbb{R}^3}\left(\sqrt{\mu_1}f_1(x,v,t)m_1|v|^2 + \sqrt{\mu_2}f_2(x,v,t)m_2|v|^2\right) dvdx =0.
\end{split}
\end{align}

\section{Dissipative property of the linearized relaxation operator}
In this part, we investigate the dissipative property of the linearized multi-component relaxation operator. For simplicity of the notation, we denote the linear operator and the nonlinear perturbation as the vector forms:
\begin{align*}
L_1&=L_{11}(f_1)+L_{12}(f_1,f_2), \cr 
L_2&=L_{22}(f_2)+L_{21}(f_1,f_2),
\end{align*}
and
\begin{align*}
\Gamma_1&=\Gamma_{11}(f_1)+\Gamma_{12}(f_1,f_2), \cr 
\Gamma_2&=\Gamma_{22}(f_2)+\Gamma_{21}(f_1,f_2),
\end{align*}
then we can write \eqref{pertf1} and \eqref{pertf2} as
\begin{align}\label{pertff}
\begin{split}
(\partial_t +v\cdot \nabla_x)(f_1,f_2)&=L(f_1,f_2)+\Gamma(f_1,f_2),
\end{split}
\end{align} 
where $L(f_1,f_2)=(L_1,L_2)$ and $\Gamma(f_1,f_2)=(\Gamma_1,\Gamma_2)$.
We also define the following $6$-dimensional orthonormal basis:
\begin{align*}
\begin{split}
&E_1= \frac{1}{\sqrt{n_{10}}}(\sqrt{\mu_1},0), \quad E_2= \frac{1}{\sqrt{n_{20}}}(0,\sqrt{\mu_2}),\cr
&E_i= \frac{1}{\sqrt{m_1n_{10}+m_2n_{20}}}\left(m_1v_{i-2}\sqrt{\mu_1},m_2v_{i-2}\sqrt{\mu_2} \right),\quad (i=3,4,5), \cr
&E_6 =\frac{1}{\sqrt{6n_{10}+6n_{10}}}\left((m_1|v|^2-3)\sqrt{\mu_1},(m_2|v|^2-3)\sqrt{\mu_2}\right).
\end{split}
\end{align*}
We also denote $E_i=(E_i^1,E_i^2)$ for $i=1,\cdots,6$. The macroscopic projection operator for mixture can be written as
\begin{align*}
P(f_1,f_2) = \sum_{1\leq i \leq 6}\langle (f_1,f_2), E_i  \rangle_{L^2_v}E_i.
\end{align*}
The following is the main result of this section.
\begin{proposition}\label{dissipation} We have the following dissipation property for the linear operator $L$:
\begin{align*}
\langle L(f_1,f_2),(f_1,f_2)\rangle_{L^2_{x,v}}&\leq - (n_{10}+n_{20})\Big( \max\{\delta,\omega\}\| (I-P_1,I-P_2)(f_1,f_2) \|_{L^2_{x,v}}^2 \cr
&\quad + \min\left\{(1-\delta),(1-\omega) \right\}\| (I-P)(f_1,f_2)\|_{L^2_{x,v}}^2 \Big).
\end{align*}
\end{proposition}
\begin{proof}
By an explicit computation, we have 
\begin{align}\label{LtoP}
\begin{split}
\langle L(f_1,f_2) &,(f_1,f_2)\rangle_{L^2_{x,v}}= 
\langle L_1f_1 , f_1\rangle_{L^2_{x,v}}+\langle L_2f_2 , f_2\rangle_{L^2_{x,v}} \cr
&= -(n_{10}+n_{20})\| (I-P_1,I-P_2)(f_1,f_2)\|_{L^2_{x,v}} +\langle L_{12}^2,f_1 \rangle_{L^2_{x,v}}+\langle L_{21}^2,f_2 \rangle_{L^2_{x,v}}.
\end{split}
\end{align}
We decompose the proof in the following 4-steps.\\
({\bf Step 1:}) We consider the dissipation from the momentum and temperature interchange part of the inter-species linearized relaxation operator. We claim that 
\begin{align*}
\langle L_{12}^2,f_1 \rangle_{L^2_v}+\langle L_{21}^2,f_2 \rangle_{L^2_v} \leq 0,
\end{align*}
and the equality holds if and only if 
\begin{align*}
\frac{1}{n_{10}}\int_{\mathbb{R}^3} f_1v\sqrt{\mu_1}dv=\frac{1}{n_{20}}\int_{\mathbb{R}^3} f_2v\sqrt{\mu_2}dv,
\end{align*}
and
\begin{align*}
\frac{1}{n_{10}}\int_{\mathbb{R}^3} f_1(m_1|v|^2-3)\sqrt{\mu_1}dv=\frac{1}{n_{20}}\int_{\mathbb{R}^3} f_2(m_2|v|^2-3)\sqrt{\mu_2}dv.
\end{align*}
\noindent$\bullet$ Proof of the claim:
By the definition of $L_{12}^2$ in \eqref{def_AL1}, we have
\begin{align*}
\langle  L_{12}^2, f_1 \rangle_{L^2_v}&=  (1-\delta) \sum_{2\leq i \leq 4}  \left(\sqrt{\frac{n_{10}}{n_{20}}}\sqrt{\frac{m_1}{m_2}}\langle f_2, e_{2i} \rangle_{L^2_v}-\langle f_1, e_{1i} \rangle_{L^2_v}\right)\langle f_1, e_{1i} \rangle_{L^2_v}n_{20} \cr
&\quad+ (1-\omega) \left(\sqrt{\frac{n_{10}}{n_{20}}}\langle  f_2, e_{25} \rangle_{L^2_v}-\langle  f_1, e_{15} \rangle_{L^2_v}\right)\langle  f_1, e_{15} \rangle_{L^2_v} n_{20} \cr
&= I_1 + I_2.
\end{align*}
Similarly,
\begin{align*}
\langle  L_{21}^2,  f_2 \rangle_{L^2_v} &= \frac{m_1}{m_2}(1-\delta)\sum_{2\leq i \leq 4}\left(\sqrt{\frac{n_{20}}{n_{10}}}\sqrt{\frac{m_2}{m_1}}\langle  f_1, e_{1i} \rangle_{L^2_v}-\langle f_2, e_{2i} \rangle_{L^2_v}\right)\langle  f_2, e_{2i} \rangle_{L^2_v}n_{10} \cr
&\quad + \frac{m_1}{m_2}(1-\omega)\left(\sqrt{\frac{n_{20}}{n_{10}}}\langle  f_1, e_{15} \rangle_{L^2_v}-\langle  f_2, e_{25} \rangle_{L^2_v}\right)\langle  f_2, e_{25} \rangle_{L^2_v}n_{10} \cr
&= I_3 + I_4.
\end{align*}
By an explicit computation, we have 
\begin{align}\label{1+3}
\begin{split}
I_1+I_3 
&= -(1-\delta) n_{20} \sum_{2\leq i \leq 4}  \left(\sqrt{\frac{n_{10}}{n_{20}}}\sqrt{\frac{m_1}{m_2}}\langle  f_2, e_{2i} \rangle_{L^2_v}-\langle  f_1, e_{1i} \rangle_{L^2_v}\right)^2 \cr
&= -(1-\delta) m_1n_{10}n_{20} \left(\frac{1}{n_{20}}\int_{\mathbb{R}^3} f_2v\sqrt{\mu_2}dv-\frac{1}{n_{10}}\int_{\mathbb{R}^3} f_1v\sqrt{\mu_1}dv \right)^2 \leq 0 ,
\end{split}
\end{align}
and 
\begin{align}\label{2+4}
\begin{split}
I_2&+I_4 = -(1-\omega)n_{20} \left(\sqrt{\frac{n_{10}}{n_{20}}}\langle  f_2, e_{25} \rangle_{L^2_v}-\langle  f_1, e_{15} \rangle_{L^2_v}\right)^2 \cr
&=-(1-\omega)\frac{n_{10}n_{20}}{6} \left(\frac{1}{n_{20}}\int_{\mathbb{R}^3} f_2(m_2|v|^2-3)\sqrt{\mu_2}dv-\frac{1}{n_{10}}\int_{\mathbb{R}^3} f_1(m_1|v|^2-3)\sqrt{\mu_1}dv\right)^2 \cr
&\leq 0 .
\end{split}
\end{align}
which proves the claim of this step. \newline
({\bf Step 2:}) To estimate the gap of the macroscopic projection $(P_1,P_2)$ with $P$, we compute the following term:
\begin{align*}
\|(P_1,P_2)(f_1,f_2)-P(f_1,f_2) \|_{L^2_{x,v}}^2.
\end{align*}
We note that the element of $(P_1,P_2)(f_1,f_2)$ can be written as the linear combination of the following $10$-dimensional basis:
\begin{align*}
\{(\sqrt{\mu_1},0),(0,\sqrt{\mu_2}),(v\sqrt{\mu_1},0),(0,v\sqrt{\mu_2}), \left(|v|^2\sqrt{\mu_1},0\right),\left(0,|v|^2\sqrt{\mu_2}\right)\}
\end{align*}
so that $(P_1,P_2)P = P$. Therefore,
\begin{align*}
\|(P_1,P_2)(f_1,f_2)-P(f_1,f_2) \|_{L^2_{x,v}}^2 = \|(P_1,P_2)(f_1,f_2)\|_{L^2_{x,v}}^2-\|P(f_1,f_2) \|_{L^2_{x,v}}^2.
\end{align*}
Since we have 
\begin{align*}
\int_{\mathbb{R}^3}|P_kf_k|^2 dv &= \frac{1}{n_{k0}}\left(\int_{\mathbb{R}^3}f_k\sqrt{\mu_k} dv\right)^2 + \frac{m_k}{n_{k0}}\left(\int_{\mathbb{R}^3}f_kv\sqrt{\mu_k} dv\right)^2 \cr
& \quad + \frac{1}{6n_{k0}}\left(\int_{\mathbb{R}^3}f_k(m_k|v|^2-3)\sqrt{\mu_k} dv\right)^2,
\end{align*}
and 
\begin{align*}
\int_{\mathbb{R}^3}|P(f_1,f_2)|^2 dv &= \frac{1}{n_{10}}\left(\int_{\mathbb{R}^3}f_1\sqrt{\mu_1} dv\right)^2 + \frac{1}{n_{20}}\left(\int_{\mathbb{R}^3}f_2\sqrt{\mu_2} dv\right)^2 \cr
&\quad + \frac{1}{m_1n_{10}+m_2n_{20}}\left(\int_{\mathbb{R}^3}f_1m_1v\sqrt{\mu_1} dv+\int_{\mathbb{R}^3}f_2m_2v\sqrt{\mu_2} dv\right)^2 \cr
&\quad + \frac{1}{6n_{10}+6n_{20}}\left(\int_{\mathbb{R}^3} f(m_1|v|^2-3)\sqrt{\mu_1} dv+\int_{\mathbb{R}^3} f(m_2|v|^2-3)\sqrt{\mu_2} dv\right)^2,
\end{align*}
which follows directly from explicit computations, we can write  
\begin{align*}
\|(P_1f_1,P_2f_2)\|_{L^2_{x,v}}^2-\|P(f_1,f_2) \|_{L^2_{x,v}}^2= II_1+II_2,
\end{align*}
where
\begin{align}\label{II_1}
\begin{split}
II_1 &= \frac{1}{m_1n_{10}}\left(\int_{\mathbb{R}^3}f_1m_1v\sqrt{\mu_1} dv\right)^2+\frac{1}{m_2n_{20}}\left(\int_{\mathbb{R}^3}f_2m_2v\sqrt{\mu_2} dv\right)^2 \cr
&\quad -\frac{1}{m_1n_{10}+m_2n_{20}}\left(\left(\int_{\mathbb{R}^3}f_1m_1v\sqrt{\mu_1} dv+\int_{\mathbb{R}^3}f_2m_2v\sqrt{\mu_2} dv\right)^2\right)\\
&=\frac{1}{m_1n_{10}+m_2n_{20}} \left[ \sqrt{\frac{m_2n_{20}}{m_1n_{10}}}\int_{\mathbb{R}^3}f_1m_1v\sqrt{\mu_1} dv-\sqrt{\frac{m_1n_{10}}{m_2n_{20}}}\int_{\mathbb{R}^3}f_2m_2v\sqrt{\mu_2} dv \right]^2
\end{split}
\end{align}
and
\begin{align}\label{II_2}
\begin{split}
II_2&= \frac{1}{6n_{10}}\left(\int_{\mathbb{R}^3}f_1(m_1|v|^2-3)\sqrt{\mu_1} dv\right)^2+\frac{1}{6n_{20}}\left(\int_{\mathbb{R}^3}f_2(m_2|v|^2-3)\sqrt{\mu_2} dv\right)^2 \cr
&\quad - \frac{1}{6n_{10}+6n_{20}}\left(\left(\int_{\mathbb{R}^3} f(m_1|v|^2-3)\sqrt{\mu_1} dv+\int_{\mathbb{R}^3} f(m_2|v|^2-3)\sqrt{\mu_2} dv\right)^2\right)\\
&= \frac{1}{6n_{10}+6n_{20}}\left[\sqrt{\frac{n_{20}}{n_{10}}} \int_{\mathbb{R}^3}f_1(m_1|v|^2-3)\sqrt{\mu_1} dv-\sqrt{\frac{n_{10}}{n_{20}}}\int_{\mathbb{R}^3}f_2(m_2|v|^2-3)\sqrt{\mu_2} dv\right]^2.
\end{split}
\end{align}

\noindent({\bf Step 3:}) In this step, we compare $\langle L_{12}^2,f_1 \rangle_{L^2_{x,v}}+\langle L_{21}^2,f_2 \rangle_{L^2_{x,v}}$ with $\|(P_1,P_2)(f_1,f_2)-P(f_1,f_2) \|_{L^2_{x,v}}^2$ computed in (Step 1) and (Step 2), respectively. We claim that
\begin{align}\label{claimP}
	\begin{split}
		\langle L_{12}^2,f_1 \rangle_{L^2_{x,v}}+\langle L_{21}^2,f_2 \rangle_{L^2_{x,v}} &\leq -\min\left\{(1-\delta),(1-\omega) \right\}	(n_{10}+n_{20}) \cr
		&\quad \times\left( \|(P_1,P_2)(f_1,f_2)\|_{L^2_{x,v}}^2-\|P(f_1,f_2) \|_{L^2_{x,v}}^2\right).
	\end{split}
\end{align}
which is equivalent to 
\begin{align}\label{equivalent}
(n_{10}+n_{20})\left(II_1+II_2\right)  \leq -\max\left\{\frac{1}{1-\delta},\frac{1}{1-\omega} \right\} \left[(I_1+I_3)+(I_2+I_4)\right],
\end{align}
where $I_i~(i=1,2,3,4)$ are defined in Step 1, and $II_i~(i=1,2)$ are defined in \eqref{II_1} and \eqref{II_2}.
We first compare $II_2$ with $I_2+I_4$. Multiplying $(n_{10}+n_{20})$ on \eqref{II_2} yields 
\begin{align*}
(n_{10}+n_{20})II_2= \frac{1}{6}\left[\sqrt{\frac{n_{20}}{n_{10}}} \int_{\mathbb{R}^3}f_1(m_1|v|^2-3)\sqrt{\mu_1} dv-\sqrt{\frac{n_{10}}{n_{20}}}\int_{\mathbb{R}^3}f_2(m_2|v|^2-3)\sqrt{\mu_2} dv\right]^2,
\end{align*}
which is equal to $-\frac{1}{1-\omega}(I_2+I_4)$ by \eqref{2+4}:
\begin{align}\label{es1}
(n_{10}+n_{20})II_2 = -\frac{1}{1-\omega}(I_2+I_4).
\end{align}
Secondly, we compare $II_1$ with $I_1+I_3$. We multiply $(n_{10}+n_{20})$ on \eqref{II_1}:
\begin{align*}
(n_{10}+n_{20})II_1 &=\frac{n_{10}+n_{20}}{m_1n_{10}+m_2n_{20}} \left[ \sqrt{\frac{m_2n_{20}}{m_1n_{10}}}\int_{\mathbb{R}^3}f_1m_1v\sqrt{\mu_1} dv-\sqrt{\frac{m_1n_{10}}{m_2n_{20}}}\int_{\mathbb{R}^3}f_2m_2v\sqrt{\mu_2} dv \right]^2 \cr
&\leq \frac{1}{m_2}\left[ \sqrt{\frac{m_2n_{20}}{m_1n_{10}}}\int_{\mathbb{R}^3}f_1m_1v\sqrt{\mu_1} dv-\sqrt{\frac{m_1n_{10}}{m_2n_{20}}}\int_{\mathbb{R}^3}f_2m_2v\sqrt{\mu_2} dv \right]^2
\end{align*}
where we used the assumption $m_1\geq m_2$. From \eqref{1+3}, we compute
\begin{align*}
-m_2(I_1+I_3) &= (1-\delta) m_1m_2n_{10}n_{20} \left(\frac{1}{n_{20}}\int_{\mathbb{R}^3} f_2v\sqrt{\mu_2}dv-\frac{1}{n_{10}}\int_{\mathbb{R}^3} f_1v\sqrt{\mu_1}dv \right)^2 
\end{align*}
which means that
\begin{align}\label{es2}
(n_{10}+n_{20})II_1 \leq -\frac{1}{1-\delta}(I_1+I_3).
\end{align}
Combining the estimates \eqref{es1} and \eqref{es2} yields the desired estimate \eqref{equivalent}.\\
({\bf Step 4:}) Finally, we go back to the estimate \eqref{LtoP}. Applying \eqref{claimP} on \eqref{LtoP} yields 
\begin{multline*}
\langle L(f_1,f_2),(f_1,f_2)\rangle_{L^2_{x,v}}\leq (n_{10}+n_{20})\left(\|(P_1,P_2)(f_1,f_2)\|_{L^2_{x,v}}^2-\|(f_1,f_2)\|_{L^2_{x,v}}^2\right) \cr
- \min\left\{(1-\delta),(1-\omega) \right\}(n_{10}+n_{20})\left( \|(P_1,P_2)(f_1,f_2)\|_{L^2_{x,v}}^2-\|P(f_1,f_2) \|_{L^2_{x,v}}^2\right).
\end{multline*}
So that,
\begin{align*}
\frac{\langle L(f_1,f_2),(f_1,f_2)\rangle_{L^2_{x,v}}}{n_{10}+n_{20}}&\leq -\|(f_1,f_2)\|_{L^2_{x,v}}^2 + \max\{\delta,\omega\}\|(P_1,P_2)(f_1,f_2)\|_{L^2_{x,v}}^2 \cr
& \quad +\min\left\{(1-\delta),(1-\omega) \right\}\|P(f_1,f_2) \|_{L^2_{x,v}}^2.
\end{align*}
Finally, by splitting $1=\max\{\delta,\omega\}+\min\left\{(1-\delta),(1-\omega) \right\}$ on the coefficient of $\|(f_1,f_2)\|_{L^2}^2$, we conclude that
\begin{align*}
\langle L(f_1,f_2),(f_1,f_2)\rangle_{L^2_{x,v}}&\leq - (n_{10}+n_{20})\Big( \max\{\delta,\omega\}\| (I-P_1,I-P_2)(f_1,f_2) \|_{L^2_{x,v}}^2 \cr
&\quad + \min\left\{(1-\delta),(1-\omega) \right\}\| (I-P)(f_1,f_2)\|_{L^2_{x,v}}^2 \Big).
\end{align*}
\end{proof}

\begin{lemma} The kernel of the linear operator $L$ satisfies
\begin{align*}
Ker L 
&= span\{(\sqrt{\mu_1},0),(0,\sqrt{\mu_2}),\cr
&\quad (m_1v\sqrt{\mu_1},m_2v\sqrt{\mu_2}),\left((m_1|v|^2-3)\sqrt{\mu_1},(m_2|v|^2-3)\sqrt{\mu_2}\right)\}.
\end{align*}
\end{lemma}
\begin{proof}
We prove the following equivalence condition.
\begin{align*}
\langle L(f_1,f_2),(f_1,f_2)\rangle_{L^2_{x,v}}=0 \qquad \Leftrightarrow  \qquad L(f_1,f_2)=0.
\end{align*}
($\Leftarrow$) This is trivial.\newline
($\Rightarrow$) By Proposition \ref{dissipation}, $\langle L(f_1,f_2),(f_1,f_2)\rangle_{L^2_{x,v}}=0$ implies $(f_1,f_2)=P(f_1,f_2)$. Now it is enough to show that $L(P(f_1,f_2))=0$. 
By direct computation,
\begin{align*}
L(P(f_1,f_2)) &= (n_{10}+n_{20})((P_1,P_2)(P(f_1,f_2))-P(f_1,f_2))\cr
&\quad +(L_{12}^2(Pf)+L_{21}^2(Pf)).
\end{align*}
The first term is equal to $0$ since $(P_1,P_2)P = P$. From (Step 1) of Proposition \ref{dissipation}, we can observe that $A_1=A_2=0$ implies $L_{12}^2=L_{21}^2=0$ where
\begin{align*}
A_1&= \frac{1}{n_{10}}\int_{\mathbb{R}^3} f_1v\sqrt{\mu_1}dv-\frac{1}{n_{20}}\int_{\mathbb{R}^3} f_2v\sqrt{\mu_2}dv,\cr
A_2&= \frac{1}{n_{10}}\int_{\mathbb{R}^3} f_1(m_1|v|^2-3)\sqrt{\mu_1}dv-\frac{1}{n_{20}}\int_{\mathbb{R}^3} f_2(m_2|v|^2-3)\sqrt{\mu_2}dv.
\end{align*}
Thus we want to prove that $A_1=A_2=0$ when $(f_1,f_2)= P(f_1,f_2) = \sum_{1\leq k \leq 6}\langle (f_1,f_2), E_k  \rangle_{L^2_v}E_k$.
From the orthogonality of the basis $E_k^1$ with $v_i\sqrt{\mu_1}$, 
\begin{align*}
A_1 &= \frac{1}{n_{10}}\int_{\mathbb{R}^3} \sum_{1\leq k \leq 6}\left[\langle (f_1,f_2), E_k  \rangle_{L^2_v}E_k^1\right]v_i\sqrt{\mu_1}dv -\frac{1}{n_{20}}\int_{\mathbb{R}^3} \sum_{1\leq k \leq 6}\left[\langle (f_1,f_2), E_k  \rangle_{L^2_v}E_k^2\right]v_i\sqrt{\mu_2}dv \cr
&= \langle (f_1,f_2), E_{i+2}  \rangle_{L^2_v} \left(\frac{1}{n_{10}}\int_{\mathbb{R}^3}E_{i+2}^1v_i\sqrt{\mu_1}dv-\frac{1}{n_{20}}\int_{\mathbb{R}^3} E_{i+2}^2v_i\sqrt{\mu_2}dv\right),
\end{align*}
for $i=1,2,3$. By definition of $E_{i+2}$, we have
\begin{align*}
A_1 &=\frac{\langle (f_1,f_2), E_{i+2}  \rangle_{L^2_v}}{\sqrt{m_1n_{10}+m_2n_{20}}} \left( \frac{1}{n_{10}}\int_{\mathbb{R}^3}m_1v_i^2\mu_1dv-\frac{1}{n_{20}}\int_{\mathbb{R}^3} m_2v_i^2\mu_2dv\right) =0 .
\end{align*}
Similarly, we compute
\begin{align*}
A_2 &= \frac{1}{n_{10}}\int_{\mathbb{R}^3} \sum_{1\leq k \leq 6}\left[\langle (f_1,f_2), E_k  \rangle_{L^2_v}E_k^1\right](m_1|v|^2-3)\sqrt{\mu_1}dv\cr
&\quad -\frac{1}{n_{20}}\int_{\mathbb{R}^3} \sum_{1\leq k \leq 6}\left[\langle (f_1,f_2), E_k  \rangle_{L^2_v}E_k^2\right](m_2|v|^2-3)\sqrt{\mu_2}dv \cr
&= \frac{\langle (f_1,f_2), E_6  \rangle_{L^2_v}}{\sqrt{6n_{10}+6n_{10}}}\left(\frac{1}{n_{10}}\int_{\mathbb{R}^3} (m_1|v|^2-3)^2\mu_1dv-\frac{1}{n_{20}}\int_{\mathbb{R}^3} (m_2|v|^2-3)^2\mu_2dv \right) \cr
&=0 ,
\end{align*}
where we used
\begin{align*}
\int_{\mathbb{R}^3}(m_i|v|^2-3)^2\mu_i dv &= \int_{\mathbb{R}^3}(m_i^2|v|^4-6m_i|v|^2+9)\mu_i dv 
= 6n_{i0}.
\end{align*}
Thus, $L_{12}^2(Pf)=L_{21}^2(Pf)=0$. Therefore, we conclude that $L(P(f_1,f_2))=0$ and the kernel of $L$ is spanned by the basis of $P$. This completes the proof. 
\end{proof}
\begin{remark} Note that in the extreme cases $\delta=1$ or $\omega=1$, we have 
\begin{itemize}
\item For $\delta=1$ and $0\leq \omega <1$
\begin{align*}
Ker L 
&= span\{(\sqrt{\mu_1},0),(0,\sqrt{\mu_2}),(v\sqrt{\mu_1},0),(0,v\sqrt{\mu_2})\cr
&\quad \left((m_1|v|^2-3)\sqrt{\mu_1},(m_2|v|^2-3)\sqrt{\mu_2}\right)\}.
\end{align*}
\item For $0\leq \delta <1$ and $\omega=1$  
\begin{align*}
Ker L 
&= span\{(\sqrt{\mu_1},0),(0,\sqrt{\mu_2}),(m_1v\sqrt{\mu_1},m_2v\sqrt{\mu_2}),\cr
&\quad \left(|v|^2\sqrt{\mu_1},0\right),\left(0,|v|^2\sqrt{\mu_2}\right)\}.
\end{align*}
\item For $\delta=\omega=1$
\begin{align*}
Ker L 
&= span\{(\sqrt{\mu_1},0),(0,\sqrt{\mu_2}),(v\sqrt{\mu_1},0),(0,v\sqrt{\mu_2})\cr
&\quad \left(|v|^2\sqrt{\mu_1},0\right),\left(0,|v|^2\sqrt{\mu_2}\right)\}.
\end{align*}
\end{itemize}
However, since $\delta=1$ or $\omega=1$ corresponds respectively to the cases where no interchange of momentum or temperature occurs. We exclude the cases in the paper sequel.
\end{remark}

\section{Local existence} 
In this section, we prove the local-in-time existence of the mixture BGK model.  We start with estimates of the macroscopic fields.
\subsection{Estimate of the macroscopic fields}
\begin{lemma}\label{macro esti} For sufficiently small $\mathcal{E}(t)$, there exists a positive constant $C>0$, such that
\begin{align*}
&(1)~ |n_{k\theta}(x,t)-n_{k0}|\leq C\sqrt{\mathcal{E}(t)},\cr
&(2)~	|U_{ij\theta}(x,t)|\leq C\sqrt{\mathcal{E}(t)},	\cr
&(3)~	|T_{ij\theta}(x,t)-1| \leq C\sqrt{\mathcal{E}(t)},
\end{align*}
for $k=1,2$ and $(i,j)=(1,2)$ or $(2,1)$.
\end{lemma} 
\begin{proof}
We recall the estimates for the mono-species macroscopic fields in \cite{Yun1}:
\begin{align*}
 |n_{k\theta}(x,t)-n_{k0}|,\quad|U_{k\theta}(x,t)|,\quad|T_{k\theta}(x,t)-1| \leq C\sqrt{\mathcal{E}(t)}.
\end{align*}
Therefore, from the definition of $U_{12\theta}$, $U_{21\theta}$, $T_{12\theta}$, and $T_{21\theta}$ in \eqref{12theta}, we have 
\begin{align*}
|U_{12\theta}|&\leq\delta |U_{1\theta}| + (1-\delta)|U_{2\theta}| \leq C\sqrt{\mathcal{E}(t)}, \cr
|U_{21\theta}|&\leq\frac{m_1}{m_2}(1-\delta)|U_{1\theta}| + \left(1-\frac{m_1}{m_2}(1-\delta)\right)|U_{2\theta}| \leq C\sqrt{\mathcal{E}(t)}, \cr
|T_{12\theta}|&= \omega |T_{1\theta}| + (1-\omega)|T_{2\theta}| + \gamma |U_{2\theta}-U_{1\theta}|^2 \leq C\sqrt{\mathcal{E}(t)}+C\mathcal{E}(t), \cr
\end{align*}
and
\begin{align*}
|T_{21\theta}|&= (1-\omega) |T_{1\theta}| + \omega |T_{2\theta}| +\left(\frac{1}{3}m_1(1-\delta)\left(\frac{m_1}{m_2}(\delta-1)+1+\delta\right)-\gamma\right) |U_{2\theta}-U_{1\theta}|^2 \cr
& \leq C\sqrt{\mathcal{E}(t)}+C\mathcal{E}(t),
\end{align*}
for sufficiently small $\mathcal{E}(t)$.
\end{proof}

\begin{lemma}\label{macro diff} For $|\alpha|\geq 1$ and sufficiently small $\mathcal{E}(t)$, there exists a positive constant $C_{\alpha}>0$, such that
\begin{align*}
&(1)~ |\partial^{\alpha}n_{k\theta}(x,t)|\leq C_{\alpha}\| \partial^{\alpha}f_k \|_{L^2_v}, \cr
&(2)~ |\partial^{\alpha}U_{ij\theta}(x,t)|\leq C_{\alpha} \sum_{|\alpha_1|\leq|\alpha|}\| \partial^{\alpha_1}f_k \|_{L^2_v},	\cr
&(3)~ |\partial^{\alpha}T_{ij\theta}(x,t)| \leq C_{\alpha} \sum_{|\alpha_1|\leq|\alpha|}\| \partial^{\alpha_1}f_k \|_{L^2_v}+C_{\alpha}\sum_{|\alpha_1|\leq|\alpha|}\| \partial^{\alpha_1}(f_1,f_2) \|_{L^2_v}^2,
\end{align*}
for $k=1,2$ and $(i,j)=(1,2)$ or $(2,1)$.
\end{lemma}
\begin{proof}
We recall \eqref{12theta} and use the following estimates from \cite{Yun1}:
\begin{align}\label{preresult}
\begin{split}
&|\partial^{\alpha}n_{k\theta}(x,t)|,~|\partial^{\alpha}U_{k\theta}(x,t)|,
~|\partial^{\alpha}T_{k\theta}(x,t)| \leq C_{\alpha} \sum_{|\alpha_1|\leq|\alpha|}\| \partial^{\alpha_1}f_k \|_{L^2_v}. \quad (k=1,2),
\end{split}
\end{align}
to get
\begin{align*}
|\partial^{\alpha}U_{12\theta}|&\leq\delta |\partial^{\alpha}U_{1\theta}| + (1-\delta)|\partial^{\alpha}U_{2\theta}| \leq C_{\alpha} \sum_{|\alpha_1|\leq|\alpha|}\|\partial^{\alpha_1}(f_1,f_2) \|_{L^2_v}, \cr
|\partial^{\alpha}U_{21\theta}|&\leq\frac{m_1}{m_2}(1-\delta)|\partial^{\alpha}U_{1\theta}| + \left(1-\frac{m_1}{m_2}(1-\delta)\right)|\partial^{\alpha}U_{2\theta}| \leq C_{\alpha} \sum_{|\alpha_1|\leq|\alpha|}\|\partial^{\alpha_1}(f_1,f_2) \|_{L^2_v},
\end{align*}
and
\begin{align*}
|\partial^{\alpha}T_{12\theta}|&= \omega |\partial^{\alpha}T_{1\theta}| + (1-\omega)|\partial^{\alpha}T_{2\theta}| + \gamma \partial^{\alpha}|U_{2\theta}-U_{1\theta}|^2, \cr
|\partial^{\alpha}T_{21\theta}|&= (1-\omega) |\partial^{\alpha}T_{1\theta}| + \omega |\partial^{\alpha}T_{2\theta}| \cr
&\quad +\left(\frac{1}{3}m_1(1-\delta)\left(\frac{m_1}{m_2}(\delta-1)+1+\delta\right)-\gamma\right) \partial^{\alpha}|U_{2\theta}-U_{1\theta}|^2.
\end{align*}
Then by Young's inequality and using $(\ref{preresult})_2$, we have 
\begin{align*}
\partial^{\alpha}|U_{2\theta}-U_{1\theta}|^2 &= \sum_{\alpha_1+\alpha_2=\alpha}2\partial^{\alpha_1}(U_{2\theta}-U_{1\theta})\cdot \partial^{\alpha_2}(U_{2\theta}-U_{1\theta}) \cr
&\leq C_{\alpha}\sum_{|\alpha_1|\leq|\alpha|}\|\partial^{\alpha_1}(f_1,f_2) \|_{L^2_v}^2,
\end{align*}
which gives desired result.
\end{proof}

\subsection{Estimate of the nonlinear term}
We now consider the estimates of nonlinear perturbation $\Gamma$.

\begin{lemma}\label{nonlin} There exist non-negative integer $\lambda$, $\nu$, $\xi$, and general polynomial $\mathcal{P}_{lm}$ satisfying
\begin{align*}
\{\nabla^2_{(H_{1\theta},H_{2\theta})}\mathcal{M}_{ij}(\theta)\}_{l,m} =\frac{\mathcal{P}_{lm}(n_{1\theta},n_{2\theta},U_{1\theta},U_{2\theta},T_{1\theta},T_{2\theta},v-U_{ij\theta})}{n_{1\theta}^{\lambda}n_{2\theta}^{\nu}T_{ij\theta}^{\xi}}\mathcal{M}_{ij}(\theta),
\end{align*}
where $\mathcal{P}_{lm}(x_1,\cdots,x_n) = \sum_k a_k x_1^{k_1}\cdots x_n^{k_n}$ and the indices $k_1,\cdots,k_n$ are non-negative integer, $ij=12$ or $ij=21$, and $1\leq l,m \leq 10$.
\end{lemma}
\begin{proof} The estimates of $\mathcal{M}_{12}(\theta)$ and $\mathcal{M}_{21}(\theta)$ are similar. We only consider the former case. We compute
\begin{align*}
\nabla_{(H_{1\theta},H_{2\theta})}\mathcal{M}_{12}(\theta)&=\left( \frac{\partial(n_{1\theta}, n_{1\theta}U_{1\theta}, G_{1\theta},n_{2\theta}, n_{2\theta}U_{2\theta}, G_{2\theta})} {\partial(n_{1\theta},U_{1\theta},T_{1\theta},n_{2\theta},U_{2\theta},T_{2\theta})} \right)^{-1} \cr
&\quad\times  \nabla_{(n_{1\theta},U_{1\theta},T_{1\theta},n_{2\theta},U_{2\theta},T_{2\theta})}\mathcal{M}_{12}(\theta).
\end{align*}
Then, as in \eqref{block inv}, we have 
\begin{align*}
\nabla_{(H_{1\theta},H_{2\theta})}\mathcal{M}_{12}(\theta)&=\left[ {\begin{array}{cccccc}
		J^{-1}_{1\theta} & 0 \\
		0 & J^{-1}_{2\theta}
\end{array} } \right] \times  \nabla_{(n_{1\theta},U_{1\theta},T_{1\theta},n_{2\theta},U_{2\theta},T_{2\theta})}\mathcal{M}_{12}(\theta) \cr
&= \left[ {\begin{array}{cccccc}
		J_{1\theta}^{-1}\nabla_{(n_{1\theta},U_{1\theta},T_{1\theta})}\mathcal{M}_{12}(\theta)  \\
		J_{2\theta}^{-1}\nabla_{(n_{2\theta},U_{2\theta},T_{2\theta})}\mathcal{M}_{12}(\theta)
\end{array} } \right] .
\end{align*}
Applying the same process once more time, we get 
\begin{align*}
\nabla^2_{(H_{1\theta},H_{2\theta})}\mathcal{M}_{12}(\theta)&=  \left[ {\begin{array}{cccccc}
		J^{-1}_{1\theta} & 0 \\
		0 & J^{-1}_{2\theta}
\end{array} } \right]  \cr
&\quad \times \nabla_{(n_{1\theta},U_{1\theta},T_{1\theta},n_{2\theta},U_{2\theta},T_{2\theta})}\left[ {\begin{array}{cccccc}
J_{1\theta}^{-1}\nabla_{(n_{1\theta},U_{1\theta},T_{1\theta})}\mathcal{M}_{12}(\theta)  \\
J_{2\theta}^{-1}\nabla_{(n_{2\theta},U_{2\theta},T_{2\theta})}\mathcal{M}_{12}(\theta)
\end{array} } \right],
\end{align*}
where the second line on the R.H.S. is equal to
\begin{multline*}
\left[ {\begin{array}{cccccc}
\nabla_{(n_{1\theta},U_{1\theta},T_{1\theta})}\left(J_{1\theta}^{-1}\nabla_{(n_{1\theta},U_{1\theta},T_{1\theta})}\mathcal{M}_{12}(\theta)\right) && \nabla_{(n_{1\theta},U_{1\theta},T_{1\theta})}\left(J_{2\theta}^{-1}\nabla_{(n_{2\theta},U_{2\theta},T_{2\theta})}\mathcal{M}_{12}(\theta)\right) \\
\nabla_{(n_{2\theta},U_{2\theta},T_{2\theta})}\left(J_{1\theta}^{-1}\nabla_{(n_{1\theta},U_{1\theta},T_{1\theta})}\mathcal{M}_{12}(\theta)\right)  && \nabla_{(n_{2\theta},U_{2\theta},T_{2\theta})}\left(J_{2\theta}^{-1}\nabla_{(n_{2\theta},U_{2\theta},T_{2\theta})}\mathcal{M}_{12}(\theta)\right)
\end{array} } \right].
\end{multline*}
Thus we get
\begin{align*}
\nabla^2_{(H_{1\theta},H_{2\theta})}\mathcal{M}_{12}(\theta)=\left[ {\begin{array}{cccccc}
T_{11} && T_{12} \\ T_{21} && T_{22}
\end{array} } \right],
\end{align*}
where
\begin{align*}
T_{ij}= J^{-1}_{i\theta}\nabla_{(n_{i\theta},U_{i\theta},T_{i\theta})}\left(J_{j\theta}^{-1}\nabla_{(n_{j\theta},U_{j\theta},T_{j\theta})}\mathcal{M}_{12}(\theta)\right),
\end{align*}
for $i,j=1,2$. Each $T_{ij}$ is a $5\times 5$ matrix. For simplicity, we only consider the $(1,1)$ and $(1,2)$ components of $\nabla^2_{(H_{1\theta},H_{2\theta})}\mathcal{M}_{12}(\theta)$. We can treat other components similarly. Recall that the first row of $J^{-1}_{1\theta}$ is $(1,0,0,0,0)$, so that 
\begin{align*}
\{\nabla^2_{(H_{1\theta},H_{2\theta})}\mathcal{M}_{12}(\theta)\}_{11}=\frac{\partial }{\partial n_{1\theta}}\frac{\partial \mathcal{M}_{12}(\theta)}{\partial n_{1\theta}} &= \frac{\partial }{\partial n_{1\theta}} \left(\frac{1}{n_{1\theta}}\mathcal{M}_{12}(\theta)\right) = 0.
\end{align*}
Now we consider the $(1,2)$ component of $\nabla^2_{(H_{1\theta},H_{2\theta})}\mathcal{M}_{12}(\theta)$ which is inner product of the first row of $J^{-1}_{1\theta}$ which is $(1,0,0,0,0)$, and the second column of $\nabla_{(n_{1\theta},U_{1\theta},T_{1\theta})}\{J_{1\theta}^{-1}\nabla_{(n_{1\theta},U_{1\theta},T_{1\theta})}\mathcal{M}_{12}(\theta)\}$. Thus,
we only need (1,2) component of $\nabla_{(n_{1\theta},U_{1\theta},T_{1\theta})}\{J_{1\theta}^{-1}\nabla_{(n_{1\theta},U_{1\theta},T_{1\theta})}\mathcal{M}_{12}(\theta)\}$: 
\begin{align}\label{D2M12}
\begin{split}
\{\nabla^2_{(H_{1\theta},H_{2\theta})}\mathcal{M}_{12}(\theta)\}_{12}&= \left[\nabla_{(n_{1\theta},U_{1\theta},T_{1\theta})}\{J_{1\theta}^{-1}\nabla_{(n_{1\theta},U_{1\theta},T_{1\theta})}\mathcal{M}_{12}(\theta)\}\right]_{12} \cr
&=\frac{\partial }{\partial n_{1\theta}} \left[J_{1\theta}^{-1}\nabla_{(n_{1\theta},U_{1\theta},T_{1\theta})}\mathcal{M}_{12}(\theta) \right]_2.
\end{split}
\end{align}
The second component of $\left[J_{1\theta}^{-1}\nabla_{(n_{1\theta},U_{1\theta},T_{1\theta})}\mathcal{M}_{12}(\theta) \right]$ is equal to the inner product of the second row of $J_{1\theta}^{-1}$ and $\nabla_{(n_{1\theta},U_{1\theta},T_{1\theta})}\mathcal{M}_{12}(\theta)$:
\begin{align*}
\left[J_{1\theta}^{-1}\nabla_{(n_{1\theta},U_{1\theta},T_{1\theta})}\mathcal{M}_{12}(\theta) \right]_2 &= \left(-\frac{U_{1\theta}}{n_{1\theta}},\frac{1}{n_{1\theta}},0,0,0\right)\cdot \nabla_{(n_{1\theta},U_{1\theta},T_{1\theta})}\mathcal{M}_{12}(\theta) \cr
&=-\frac{U_{1\theta}}{n_{1\theta}} \frac{\partial \mathcal{M}_{12}(\theta)}{\partial n_{1\theta}} + \frac{1}{n_{1\theta}} \frac{\partial \mathcal{M}_{12}(\theta)}{\partial U_{11\theta}}.
\end{align*}
Substituting this into \eqref{D2M12} gives 
\begin{align*}
\{\nabla^2_{(H_{1\theta},H_{2\theta})}\mathcal{M}_{12}(\theta)\}_{12} 
&=\frac{\partial }{\partial n_{1\theta}} \left(-\frac{U_{11\theta}}{n_{1\theta} }\frac{\partial \mathcal{M}_{12}(\theta)}{\partial n_{1\theta}}+\frac{1}{n_{1\theta}}\frac{\partial \mathcal{M}_{12}(\theta)}{\partial U_{11\theta}}\right) \cr
&= \frac{U_{11\theta}}{n_{1\theta}^2 }\frac{\partial \mathcal{M}_{12}(\theta)}{\partial n_{1\theta}}-\frac{U_{11\theta}}{n_{1\theta} }\frac{\partial^2 \mathcal{M}_{12}(\theta)}{\partial n_{1\theta}^2}-\frac{1}{n_{1\theta}^2}\frac{\partial \mathcal{M}_{12}(\theta)}{\partial U_{11\theta}}+\frac{1}{n_{1\theta}}\frac{\partial^2 \mathcal{M}_{12}(\theta)}{\partial n_{1\theta}\partial U_{11\theta}}.
\end{align*}
Then, from Lemma \ref{M_12 diff} (1) and (2), we have
\begin{align*}
\{\nabla^2_{(H_{1\theta},H_{2\theta})}&\mathcal{M}_{12}(\theta)\}_{12} 
= \frac{U_{11\theta}}{n_{1\theta}^3}\mathcal{M}_{12}(\theta)\cr
& + \frac{1}{n_{1\theta}}\left(\delta m_1\frac{v-U_{12\theta}}{T_{12\theta}} -2\gamma(U_{2\theta}-U_{1\theta})\left(-\frac{3}{2}\frac{1}{T_{12\theta}}+\frac{m_1|v-U_{12\theta}|^2}{2T_{12\theta}^2}\right)\right)\mathcal{M}_{12}(\theta).
\end{align*}
We observe that $(1,2)$ component of $\nabla^2_{(H_{1\theta},H_{2\theta})}\mathcal{M}_{12}(\theta)$ is expressed in the form presented in this lemma.
\end{proof}
We are now ready to estimate the nonlinear terms. The intra-species part is established in \cite{Yun1}:
\begin{lemma}\emph{\cite{Yun1}}\label{nonlin esti1} For sufficiently small $\mathcal{E}(t)$, we have the following inequality for $k=1,2$.
	\begin{align*}
	\langle \partial^{\alpha}_{\beta} \Gamma_{kk}(f_k), g\rangle_{L^2_v} \leq  C\sum_{|\alpha_1|+|\alpha_2|\leq |\alpha|}\|\partial^{\alpha_1}f_k\|_{L^2_v}\|\partial^{\alpha_2}f_k\|_{L^2_v}\|g \|_{L^2_v}.
	\end{align*}
\end{lemma} 
So we focus on the inter-species part.
\begin{lemma}\label{nonlin esti} Let $N\geq3$ and $|\alpha|+|\beta|\leq N$. For sufficiently small $\mathcal{E}(t)$, we have
\begin{align*}
\langle \partial^{\alpha}_{\beta} \Gamma_{ij}, g\rangle_{L^2_v} \leq  C\sum_{|\alpha_1|+|\alpha_2|\leq |\alpha|}\|\partial^{\alpha_1}(f_1,f_2)\|_{L^2_v}\|\partial^{\alpha_2}_{\beta}(f_1,f_2)\|_{L^2_v}\|g \|_{L^2_v},
\end{align*}
for $(i,j)=(1,2)$ or $(2,1)$.
\end{lemma}
\begin{proof}
We only consider the $\Gamma_{12}$ since the estimate of $\Gamma_{21}$ is similar.
Therefore, we focus on the estimates of the nonlinear terms $\Gamma_{12}$ and $\Gamma_{21}$. For convenience, we divide $\Gamma_{12}$ into three parts:
\begin{align*}
\Gamma_{12} = \Gamma_{12A}+\Gamma_{12B}+\Gamma_{12C}, 
\end{align*}
where
\begin{align*}
\Gamma_{12A}&= (n_2-n_{20})(P_1f_1-f_1), \cr
\Gamma_{12B}&= n_2 \frac{1}{\sqrt{\mu_1}} \int_0^1\mathcal{M}_{12}''(\theta)(1-\theta)d\theta, \cr
\Gamma_{12C}&= (n_2-n_{20})\bigg[(1-\delta) \sum_{2\leq i \leq 4}  \left(\sqrt{\frac{n_{10}}{n_{20}}}\sqrt{\frac{m_1}{m_2}}\langle f_2, e_{2i} \rangle_{L^2_v}-\langle f_1, e_{1i} \rangle_{L^2_v}\right)e_{1i} \cr
&\quad +(1-\omega) \left(\sqrt{\frac{n_{10}}{n_{20}}}\langle f_2, e_{25} \rangle_{L^2_v}-\langle f_1, e_{15} \rangle_{L^2_v}\right)e_{15} \bigg].
\end{align*}
We first write $\Gamma_{12B}$ in a concise form before we delve into the estimate.  For this, compute applying the chain rule twice on $\mathcal{M}_{ij}$:
\begin{align*}
&\mathcal{M}_{ij}''(\theta) \cr
&\quad= \frac{d}{d\theta} \bigg( \frac{d n_{\theta1}}{d \theta}\frac{d \mathcal{M}_{ij}}{d n_{\theta1}}+\frac{d (n_{\theta1}U_{\theta1})}{d \theta}\frac{d \mathcal{M}_{ij}}{d (n_{\theta1}U_{\theta1})}+\frac{d G_{\theta1}}{d \theta}\frac{d \mathcal{M}_{ij}}{d G_{\theta1}} \cr
&\qquad +\frac{d n_{\theta2}}{d \theta}\frac{d \mathcal{M}_{ij}}{d n_{\theta2}}+\frac{d (n_{\theta2}U_{\theta2})}{d \theta}\frac{d \mathcal{M}_{ij}}{d (n_{\theta2}U_{\theta2})}+\frac{d G_{\theta2}}{d \theta}\frac{d \mathcal{M}_{ij}}{d G_{\theta2}} \bigg) \cr
&\quad=(n_1-n_{10}, n_1 U_1, G_1,n_2-n_{20}, n_2 U_2, G_2)^T\left\{\nabla^2_{(n_{1\theta}, n_{1\theta}U_{1\theta}, G_{1\theta},n_{2\theta}, n_{2\theta}U_{2\theta}, G_{2\theta})}\mathcal{M}_{ij}(\theta)\right\} \cr
&\qquad\times (n_1-n_{10}, n_1 U_1, G_1,n_2-n_{20}, n_2 U_2, G_2).
\end{align*}
Therefore, if we define
\begin{align}\label{defH}
H_{k}=(n_{k}, n_{k}U_{k}, G_{k}) ,\quad \textit{and} \quad H_{k\theta}=(n_{k\theta}, n_{k\theta}U_{k\theta}, G_{k\theta}),
\end{align}
we can rewrite $\Gamma_{12B}$ as 
\begin{align*}
\Gamma_{12B} &=  \frac{n_2}{\sqrt{\mu_1}}  (H_1-H_{10},H_2-H_{20})^T  \cr
&\quad \times \int_0^1 \left\{\nabla^2_{(H_{1\theta},H_{2\theta})}\mathcal{M}_{ij}(\theta)\right\} (1-\theta)d\theta  (H_1-H_{10},H_2-H_{20}).
\end{align*}
Now we estimate each part of  $\Gamma_{12}$. \\
$\bullet$ {\bf Estimate of $\Gamma_{12A}$:} We take a derivative $\partial^{\alpha}_{\beta}$ on $\Gamma_{12A}$:
\begin{align*}
\partial^{\alpha}_{\beta}\Gamma_{12A} &=\sum_{\alpha_1+\alpha_2=\alpha} C_{\alpha_1} \partial^{\alpha_1}(n_2-n_{20})\partial^{\alpha_2}_{\beta}(P_1f_1-f_1).
\end{align*}
From \eqref{rho decomp}, we have 
\begin{align}\label{e1}
\partial^{\alpha}(n_2-n_{20}) \leq C\| \partial^{\alpha}f_2 \|_{L^2_v}.
\end{align}
For an estimate of the macroscopic projection $P_1f_1$, since $\partial_{\beta} e_{1i}$ has an exponential decay, we get
\begin{align*}
\| \partial^{\alpha}_{\beta}P_1f_1 \|_{L^2_v}= \| \partial_{\beta}P_1\partial^{\alpha}f_1 \|_{L^2_v} \leq C_{\beta}\| \partial^{\alpha}f_1 \|_{L^2_v}.
\end{align*}
Thus we have 
\begin{align}\label{P1}
\langle \partial^{\alpha}_{\beta}(P_1f_1-f_1), g \rangle_{L^2_v}
\leq C\left(\| \partial^{\alpha}f_1 \|_{L^2_v}+\| \partial^{\alpha}_{\beta}f_1 \|_{L^2_v}\right)\|g\|_{L^2_v}.
\end{align}
Combining \eqref{e1} and \eqref{P1}, we obtain 
\begin{align*}
\langle \partial_{\beta}^{\alpha} \Gamma_{12A},g \rangle_{L^2_v} \leq C\sum_{|\alpha_1|+|\alpha_2|+|\beta|\leq N}\| \partial^{\alpha_1}f_2 \|_{L^2_v}\left(\| \partial^{\alpha_2}f_1 \|_{L^2_v}+\| \partial^{\alpha_2}_{\beta}f_1 \|_{L^2_v}\right)\|g\|_{L^2_v}.
\end{align*}
$\bullet$ {\bf Estimate of $\Gamma_{12C}$:} We take a derivative $\partial_{\beta}^{\alpha}$ on $\Gamma_{12C}$: 
\begin{align*}
\partial^{\alpha}_{\beta}\Gamma_{12C}&= \sum_{\alpha_1+\alpha_2=\alpha}C_{\alpha_1} \partial^{\alpha_1}(n_2-n_{20})\cr 
&\quad \times \bigg[(1-\delta) \sum_{2\leq i \leq 4}  \left(\sqrt{\frac{n_{10}}{n_{20}}}\sqrt{\frac{m_1}{m_2}}\langle \partial^{\alpha_2}f_2, e_{2i} \rangle_{L^2_v}-\langle \partial^{\alpha_2}f_1, e_{1i} \rangle_{L^2_v}\right)\partial_{\beta}e_{1i} \cr
&\quad +(1-\omega) \left(\sqrt{\frac{n_{10}}{n_{20}}}\langle \partial^{\alpha_2}f_2, e_{25} \rangle_{L^2_v}-\langle \partial^{\alpha_2}f_1, e_{15} \rangle_{L^2_v}\right)\partial_{\beta}e_{15} \bigg].
\end{align*}
Since each $e_{1i}$ and $e_{2i}$ has exponential decay for $i=1,\cdots,5$, we can have
\begin{align}
\langle \partial^{\alpha}f_1, e_{1i} \rangle_{L^2_v} \leq C\| \partial^{\alpha}f_1 \|_{L^2_v}, \qquad \langle \partial^{\alpha}f_2, e_{2i} \rangle_{L^2_v} \leq C\| \partial^{\alpha}f_2 \|_{L^2_v},
\label{e2}
\end{align}
and
\begin{align}\label{e3}
\langle \partial_{\beta}e_{1i}, g \rangle_{L^2_v} \leq C\|g\|_{L^2_v} \qquad \langle \partial_{\beta}e_{2i}, g \rangle_{L^2_v} \leq C\|g\|_{L^2_v}.
\end{align}
Thus by using \eqref{e1}, \eqref{e2}, and \eqref{e3}, we get 
\begin{align*}
\langle \partial^{\alpha}_{\beta} \Gamma_{12C}, g\rangle_{L^2_v} \leq  C\sum_{|\alpha_1|+|\alpha_2|\leq |\alpha|}\|\partial^{\alpha_1}f_2\|_{L^2_v}\|\partial^{\alpha_2}(f_1,f_2)\|_{L^2_v}\|g \|_{L^2_v}.
\end{align*}
$\bullet$ {\bf Estimate of $\Gamma_{12B}$:} Taking $\partial^{\alpha}_{\beta}$ on $\Gamma_{12B}$ gives
\begin{align}\label{gamma12Bes}
\begin{split}
\partial^{\alpha}_{\beta}\Gamma_{12B} &= \sum_{\sum \alpha_i=\alpha}C_{\alpha_i}\partial^{\alpha_0}n_2 \partial^{\alpha_1}(H_1-H_{10},H_2-H_{20})^T \cr 
&\times \int_0^1\partial^{\alpha_2}_{\beta}\left\{\frac{1}{\sqrt{\mu_1}}\nabla^2_{(H_{1\theta},H_{2\theta})}\mathcal{M}_{12}(\theta)\right\}(1-\theta)d\theta  \partial^{\alpha_3}(H_1-H_{10},H_2-H_{20}) .
\end{split}
\end{align}
By the definition of $H_k$ in \eqref{defH}, applying \eqref{lin1} yields
\begin{align*}
\partial^{\alpha}(H_k-H_{k0})&=\partial^{\alpha}(n_k-n_{k0}, n_k U_k, G_k)= \left(  \langle \partial^{\alpha}f_k, e_{k1}  \rangle_{L^2_v}, \cdots, \langle \partial^{\alpha}f_k, e_{k5} \rangle_{L^2_v}  \right),
\end{align*}
for $k=1,2$.
Thus we have
\begin{align}\label{HkH0}
|\partial^{\alpha}(H_k-H_{k0})| \leq C \| \partial^{\alpha}f_k \|_{L^2_v}.
\end{align}
For notational simplicity, we set
\begin{align*}
A_{lm}=\int_0^1\partial^{\alpha_2}_{\beta}\left\{\frac{1}{\sqrt{\mu_1}}\nabla^2_{(H_{1\theta},H_{2\theta})}\mathcal{M}_{12}(\theta)\right\}_{l,m}(1-\theta)d\theta.
\end{align*}
Then by Lemma \ref{nonlin}, we can write it as 
\begin{align}\label{Alm}
A_{lm}=\int_0^1\partial^{\alpha_2}_{\beta}\left\{\frac{1}{\sqrt{\mu_1}} \frac{\mathcal{P}_{lm}(n_{1\theta},n_{2\theta},U_{1\theta},U_{2\theta},T_{1\theta},T_{2\theta},v-U_{12\theta})}{n_{1\theta}^{\lambda}n_{2\theta}^{\nu}T_{12\theta}^{\xi}}\mathcal{M}_{12}(\theta)\right\}(1-\theta)d\theta.
\end{align}
By the product rule, we have 
\begin{align*}
&\partial^{\alpha}_{\beta}\left\{ \frac{\mathcal{P}_{lm}(n_{1\theta},n_{2\theta},U_{1\theta},U_{2\theta},T_{1\theta},T_{2\theta},v-U_{12\theta})}{n_{1\theta}^{\lambda}n_{2\theta}^{\nu}T_{12\theta}^{\xi}}\right\} \cr
&=C_{\alpha}\sum_{\sum \alpha_i=\alpha}\bigg\{\mathcal{P}_{lm}(\partial^{\alpha_1}n_{1\theta},\partial^{\alpha_2}n_{2\theta},\partial^{\alpha_3}U_{1\theta},\partial^{\alpha_4}U_{2\theta},\partial^{\alpha_5}T_{1\theta},\partial^{\alpha_6}T_{2\theta},\partial^{\alpha_7}_{\beta}(v-U_{12\theta})) \cr
&\quad \times  \partial^{\alpha_8}\frac{1}{n_{1\theta}^{\lambda}n_{2\theta}^{\nu}T_{12\theta}^{\xi}}\bigg\} 
\end{align*}
If $|\alpha_i|\leq N-2$, then by Sobolev embedding $H^2 \subset\subset L^{\infty}$ and Lemma \ref{macro diff}, we have 
\begin{align*}
|\partial^{\alpha}n_{k\theta}(x,t)|+|\partial^{\alpha}U_{k\theta}(x,t)|+|\partial^{\alpha}T_{k\theta}(x,t)|\leq C\| \partial^{\alpha}f_k \|_{L^2_v}\leq \sqrt{\mathcal{E}(t)}.
\end{align*}
Since $N\geq 3$, there is at most one $\alpha_i$ that exceeds $N-2$. Thus, for sufficiently small $\mathcal{E}(t)$, we have
\begin{align*}
\partial^{\alpha}_{\beta}\left\{ \frac{\mathcal{P}_{lm}(n_{1\theta},n_{2\theta},U_{1\theta},U_{2\theta},T_{1\theta},T_{2\theta},v-U_{12\theta})}{n_{1\theta}^{\lambda}n_{2\theta}^{\nu}T_{12\theta}^{\xi}}\right\} \leq C\sqrt{\mathcal{E}(t)}\|\partial^{\alpha}f \|_{L^2_v} \mathcal{P}_{lm}(v).
\end{align*}
Substituting it in \eqref{Alm} yields  
\begin{align*}
A_{lm} \leq C\sqrt{\mathcal{E}(t)}\|\partial^{\alpha}f \|_{L^2_v} \mathcal{P}_{lm}(v) \partial^{\alpha}_{\beta} \exp\left(-\frac{|v-U_{12\theta}|^2}{2\frac{T_{12\theta}}{m_1}}+\frac{m_1|v|^2}{4}\right).
\end{align*}
Similarly, the derivative of the exponential part can be bounded as follows:
\begin{multline*}
\partial^{\alpha}_{\beta} \exp\left(-\frac{|v-U_{12\theta}|^2}{2\frac{T_{12\theta}}{m_1}}+\frac{m_1|v|^2}{4}\right) \cr
\leq C\sqrt{\mathcal{E}(t)}\|\partial^{\alpha}f \|_{L^2_v} \mathcal{P}_{lm}(v) \exp\left(-\frac{|v-U_{12\theta}|^2}{2\frac{T_{12\theta}}{m_1}}+\frac{m_1|v|^2}{4}\right).
\end{multline*}
By Lemma \ref{macro esti} (3), a sufficiently small $\mathcal{E}(t)$ guarantees $T_{12\theta} \leq 3/2$, so that
\begin{align}\label{Almg}
\begin{split}
\langle A_{lm} , g \rangle_{L^2_v} &\leq C \left\| P(v) \exp\left(-\frac{2m_1|v-U_{12\theta}|^2}{3}+\frac{m_1|v|^2}{2}\right) \right\|_{L^2_v}  \| g \|_{L^2_v} \cr
&\leq C \left\| P(v) \exp\left(-\frac{m_1|v-4U_{12\theta}|^2}{6}+2m_1|U_{12\theta}|^2\right) \right\|_{L^2_v}  \| g \|_{L^2_v} \cr
&\leq C\| g \|_{L^2_v},
\end{split}
\end{align}
where we used $e^{2m_1|U_{12\theta}|^2}\leq C$ for sufficiently small $\mathcal{E}(t)$. Substituting \eqref{HkH0} and \eqref{Almg} on \eqref{gamma12Bes} gives the desired result.
\end{proof}

\subsection{Local existence}
In this part, we prove the existence of a local-in-time classical solution of the mixture BGK model \eqref{CCBGK}.
\begin{theorem}\label{theo_loc_ex}
Let $F_{10}=\mu_1 + \sqrt{\mu_1} f_{10}\geq 0$ and $F_{20}= \mu_2 + \sqrt{\mu_2} f_{20} \geq 0$. There exists $ T_*>0$ and $M_0>0$ such that if $\mathcal{E}(0) \leq \frac{M_0}{2}$, then there exists a unique local-in-time solution $(F_1,F_2)$ of \eqref{CCBGK} such that 
\begin{enumerate}
\item The distribution functions $F_1(x,v,t)$ and $F_2(x,v,t)$ are non-negative.
\item The high-order energy $\mathcal{E}(t)$ is uniformly bounded:
$$ \sup_{0 \leq t \leq T_*} \mathcal{E}(t) \leq M_0.$$
\item The high-order energy is continuous in $t \in [0,T_*)$.
\item The conservation laws \eqref{conservf} hold for all $t \in [0,T_*)$.
\end{enumerate}
\end{theorem}
\begin{proof}
We define an iteration of the mixture BGK model \eqref{CCBGK} as follows:
\begin{align*}
\begin{aligned}
\partial_t F_1^{n+1}+v\cdot \nabla_xF_1^{n+1}&=n_1(F_1^n)(\mathcal{M}_{11}(F_1^n)-F_1^{n+1})\cr 
&\quad +n_2(F_2^n)(\mathcal{M}_{12}(F_1^n,F_2^n)-F_1^{n+1}), \cr
\partial_t F_2^{n+1}+v\cdot \nabla_xF_2^{n+1}&=n_2(F_2^n)(\mathcal{M}_{22}(F_2^n)-F_2^{n+1})\cr
&\quad +n_1(F_1^n)(\mathcal{M}_{21}(F_1^n,F_2^n)-F_2^{n+1}),
\end{aligned}
\end{align*}
and $F_1^{n+1}(x,v,0)=F_{10}(x,v)$ and $F_2^{n+1}(x,v,0)=F_{20}(x,v)$ for all $n \geq 0$. We start the iteration with $F_1^0(x,v,t)=F_{10}(x,v)$ and $F_2^0(x,v,t)=F_{20}(x,v)$.

We split $F_1^n=\mu_1+\sqrt{\mu_1}f_1^n$, and $F_2^n=\mu_2+\sqrt{\mu_2}f_2^n$ for all $n \in \mathbb{N}$ and use the linearization of the Maxwellian given in Proposition \ref{linearize} and Lemma \ref{lin ii} to get
\begin{align*}
\partial_t f_1^{n+1}+v\cdot \nabla_xf_1^{n+1}&=(n_{10}+n_{20})(P_1f_1^n-f_1^{n+1})+L_{12}^2(f_1^n,f_2^n)+\Gamma_{11}(f_1^n)+\Gamma_{12}(f_1^n,f_2^n), \cr
\partial_t f_2^{n+1}+v\cdot \nabla_xf_2^{n+1}&=(n_{10}+n_{20})(P_2f_2^n-f_2^{n+1})+L_{21}^2(f_1^n,f_2^n)+\Gamma_{22}(f_2^n)+\Gamma_{21}(f_1^n,f_2^n).
\end{align*}
Then the local existence can be constructed by the standard argument as in \cite{Guo VMB}.
The key ingredient is the uniform control of the high-order energy norm in each iteration step. So we only prove the following auxiliary lemma below.
\end{proof}
\begin{lemma}
Let $\mathcal{E}(0)<\frac{M_0}{2}$. Then there exists $T_*>0$ and $M_0>0$ such that $\mathcal{E}(f^n(t))<M_0$ for all $n\geq 0$ and $t\in [0,T_*].$
\end{lemma}
\begin{proof} We take $\partial^{\alpha}_{\beta}$ on each side of \eqref{pertf1} and \eqref{pertf2}:
\begin{multline*}
\partial^{\alpha}_{\beta}\partial_t f_1^{n+1}+v\cdot \nabla_x\partial^{\alpha}_{\beta}f_1^{n+1}+\sum_{i=1}^3  \partial^{\alpha+\bar{k}_i}_{\beta-k_i}f_1^{n+1}
=(n_{10}+n_{20})(\partial_{\beta}P_1\partial^{\alpha}f_1^n-\partial^{\alpha}_{\beta}f_1^{n+1})\cr
+\partial^{\alpha}_{\beta}L_{12}^2(f_1^n,f_2^n)+\partial^{\alpha}_{\beta}\Gamma_{11}(f_1^n)+\partial^{\alpha}_{\beta}\Gamma_{12}(f_1^n,f_2^n), 
\end{multline*}
and
\begin{multline*}
\partial^{\alpha}_{\beta}\partial_t f_2^{n+1}+v\cdot \nabla_x\partial^{\alpha}_{\beta}f_2^{n+1}+\sum_{i=1}^3  \partial^{\alpha+\bar{k}_i}_{\beta-k_i}f_2^{n+1}=(n_{10}+n_{20})(\partial_{\beta}P_2\partial^{\alpha}f_2^n-\partial^{\alpha}_{\beta}f_2^{n+1})\cr
+\partial^{\alpha}_{\beta}L_{21}^2(f_1^n,f_2^n)+\partial^{\alpha}_{\beta}\Gamma_{22}(f_2^n)+\partial^{\alpha}_{\beta}\Gamma_{21}(f_1^n,f_2^n),
\end{multline*}
where $k_1=(1,0,0)$, $k_2=(0,1,0)$, $k_3=(0,0,1)$, and $\bar{k}_1=(0,1,0,0)$, $\bar{k}_2=(0,0,1,0)$, $\bar{k}_3=(0,0,0,1)$. We then take the inner product with $\partial^{\alpha}_{\beta}f_1^{n+1}$: 
\begin{align}\label{eqtoest}
\begin{split}
\frac{1}{2}\frac{d}{dt}\|\partial^{\alpha}_{\beta} f_1^{n+1} \|_{L^2_{x,v}}^2&+(n_{10}+n_{20})\|\partial^{\alpha}_{\beta} f_1^{n+1} \|_{L^2_{x,v}}^2 = - \sum_{i=1}^3 \langle \partial^{\alpha+\bar{k}_i}_{\beta-k_i}f_1^{n+1}, \partial^{\alpha}_{\beta} f_1^{n+1} \rangle_{L^2_{x,v}} \cr
&\quad + 
\langle \partial_{\beta}P_1\partial^{\alpha}f_1^n,\partial^{\alpha}_{\beta} f_1^{n+1} \rangle_{L^2_{x,v}}
+\langle \partial^{\alpha}_{\beta}L_{12}^2(f_1^n,f_2^n), \partial^{\alpha}_{\beta} f_1^{n+1} \rangle_{L^2_{x,v}} \cr
&\quad +\langle \partial^{\alpha}_{\beta}\Gamma_{11}(f_1^n),\partial^{\alpha}_{\beta} f_1^{n+1} \rangle_{L^2_{x,v}}+\langle \partial^{\alpha}_{\beta}\Gamma_{12}(f_1^n,f_2^n),\partial^{\alpha}_{\beta} f_1^{n+1} \rangle_{L^2_{x,v}} \cr
&\quad = I_1+I_2+I_3+I_4+I_5.
\end{split}
\end{align}
Applying the H\"{o}lder inequality on $I_1$, we have 
\begin{align*}
I_1=\sum_{i=1}^3 \langle \partial^{\alpha+\bar{k}_i}_{\beta-k_i}f_1^{n+1}, \partial^{\alpha}_{\beta} f_1^{n+1} \rangle_{L^2_{x,v}} &\leq \sum_{i=1}^3\|\partial^{\alpha+\bar{k}_i}_{\beta-k_i}f_1^{n+1}\|_{L^2_{x,v}}\|\partial^{\alpha}_{\beta} f_1^{n+1} \|_{L^2_{x,v}} \cr
& \leq \sum_{|\alpha|+|\beta|\leq N } \|\partial^{\alpha}_{\beta} f_1^{n+1} \|_{L^2_{x,v}}^2.
\end{align*}
Since $\partial_{\beta} e_{1i}$ and $\partial_{\beta} e_{2i}$ have exponential decay, 
\begin{align*}
\| \partial_{\beta}P_1\partial^{\alpha}f_1^{n} \|_{L^2_{x,v}} \leq C_{\beta}\| \partial^{\alpha}f_1^{n} \|_{L^2_{x,v}}.
\end{align*}
Thus Young's inequality implies 
\begin{align*}
I_2=\langle \partial_{\beta}P_1\partial^{\alpha}f_1^n,\partial^{\alpha}_{\beta} f_1^{n+1} \rangle_{L^2_{x,v}} &\leq   C_{\beta} \| \partial^{\alpha}f_1^n\|_{L^2_{x,v}}^2 + C \| \partial^{\alpha}_{\beta} f_1^{n+1} \|_{L^2_{x,v}}^2. 
\end{align*}
To estimate $I_3$, we take $\partial^{\alpha}_{\beta}$ on  $L_{12}^2$:
\begin{multline*}
\partial^{\alpha}_{\beta}L_{12}^2(f_1,f_2) = n_{20} \bigg[(1-\delta) \sum_{2\leq i \leq 4}  \left(\sqrt{\frac{n_{10}}{n_{20}}}\sqrt{\frac{m_1}{m_2}}\langle \partial^{\alpha}f_2, e_{2i} \rangle_{L^2_v}-\langle \partial^{\alpha}f_1, e_{1i} \rangle_{L^2_v}\right)\partial_{\beta}e_{1i} \cr
+(1-\omega) \left(\sqrt{\frac{n_{10}}{n_{20}}}\langle \partial^{\alpha}f_2, e_{25} \rangle_{L^2_v}-\langle \partial^{\alpha}f_1, e_{15} \rangle_{L^2_v}\right)\partial_{\beta}e_{15} \bigg],
\end{multline*}
and apply the H\"{o}lder inequality: 
\begin{align*}
I_3=\langle \partial^{\alpha}_{\beta}L_{12}^2(f_1^n,f_2^n), \partial^{\alpha}_{\beta} f_1^{n+1} \rangle_{L^2_{x,v}} &\leq C\int_{\mathbb{T}^3} \left(\| \partial^{\alpha}f_2^n\|_{L^2_v} + C\| \partial^{\alpha}f_1^n\|_{L^2_v}\right) \| \partial^{\alpha}_{\beta}f_1^{n+1}\|_{L^2_v} dx \cr
&\leq C\| \partial^{\alpha}(f_1^n,f_2^n)\|_{L^2_{x,v}}  \| \partial^{\alpha}_{\beta}f_1^{n+1}\|_{L^2_{x,v}}.
\end{align*}
Since $I_4$ and $I_5$ are similar, we only consider $I_5$. Applying Lemma \ref{nonlin esti}, we have
\begin{align*}
I_5=\langle \partial^{\alpha}_{\beta}\Gamma_{12}(f_1^n,f_2^n),\partial^{\alpha}_{\beta} f_1^{n+1} \rangle_{L^2_{x,v}} &\leq C\sum_{|\alpha_1|+|\alpha_2|\leq |\alpha|} \int_{\mathbb{T}^3}\|\partial^{\alpha_1}(f_1^n,f_2^n)\|_{L^2_v} \cr
&\quad\times \|\partial^{\alpha_2}(f_1^n,f_2^n)\|_{L^2_v}\| \partial^{\alpha}_{\beta} f_1^{n+1} \|_{L^2_v}dx.
\end{align*}
Without loss of generality, we assume that $|\alpha_1|\leq |\alpha_2|$. Then the Sobolev embedding $H^2 \subset\subset L^{\infty}$ implies 
\begin{align*}
I_5&=\langle \partial^{\alpha}_{\beta}\Gamma_{12}(f_1^n,f_2^n),\partial^{\alpha}_{\beta} f_1^{n+1} \rangle_{L^2_{x,v}} \cr
&\leq C\bigg(\sum_{|\alpha_1|\leq |\alpha|} \|\partial^{\alpha_1}(f_1^n,f_2^n)\|_{L^2_{x,v}}\bigg)^2 \| \partial^{\alpha}_{\beta} f_1^{n+1} \|_{L^2_{x,v}}.
\end{align*}
Combining the estimate from $I_1$ to $I_5$, and taking $\sum_{|\alpha|+|\beta|\leq N}$ on \eqref{eqtoest}, we have
\begin{align}\label{est_dt1}
\begin{split}
\frac{1}{2}\sum_{|\alpha|+|\beta|\leq N }&\frac{d}{dt}\|\partial^{\alpha}_{\beta} f_1^{n+1} \|_{L^2_{x,v}}^2+(n_{10}+n_{20})\sum_{|\alpha|+|\beta|\leq N }\|\partial^{\alpha}_{\beta} f_1^{n+1} \|_{L^2_{x,v}}^2 \cr
&\leq C\mathcal{E}^n(t) + C \mathcal{E}^{n+1}(t) + C\sqrt{\mathcal{E}^n(t)}\sqrt{\mathcal{E}^{n+1}(t)}  +C\mathcal{E}^n(t)\sqrt{\mathcal{E}^{n+1}(t)}.
\end{split}
\end{align}
Similarly,
\begin{align}\label{est_dt2}
\begin{split}
\frac{1}{2}\sum_{|\alpha|+|\beta|\leq N }&\frac{d}{dt}\|\partial^{\alpha}_{\beta} f_2^{n+1} \|_{L^2_{x,v}}^2+\sum_{|\alpha|+|\beta|\leq N }(n_{10}+n_{20})\|\partial^{\alpha}_{\beta} f_2^{n+1} \|_{L^2_{x,v}}^2 \cr 
&\leq C\mathcal{E}^n(t) + C \mathcal{E}^{n+1}(t) + C\sqrt{\mathcal{E}^n(t)}\sqrt{\mathcal{E}^{n+1}(t)}  +C\mathcal{E}^n(t)\sqrt{\mathcal{E}^{n+1}(t)}.
\end{split}
\end{align}
Combining \eqref{est_dt1} and \eqref{est_dt2} yields 
\begin{align*}
\frac{1}{2}\frac{d}{dt}\mathcal{E}^{n+1}(t)+(n_{10}+n_{20})\mathcal{E}^{n+1}(t) &\leq C
\mathcal{E}^n(t) + C \mathcal{E}^{n+1}(t)   \cr
& \quad + C\sqrt{\mathcal{E}^n(t)}\sqrt{\mathcal{E}^{n+1}(t)}  +C\mathcal{E}^n(t)\sqrt{\mathcal{E}^{n+1}(t)}.
\end{align*}
We integrate in time to get  
\begin{align}\label{Eps_n+1}
\begin{split}
\mathcal{E}^{n+1}&(t)\leq \mathcal{E}^{n+1}(0)\cr
&+\int_0^t \left(C\mathcal{E}^n(s) +C \mathcal{E}^{n+1}(s)+ C\sqrt{\mathcal{E}^n(t)}\sqrt{\mathcal{E}^{n+1}(t)} +C\mathcal{E}^n(t)\sqrt{\mathcal{E}^{n+1}(t)}\right)ds.
\end{split}
\end{align}
We now apply an induction argument. We have $\mathcal{E}^0(0)<\frac{M_0}{2}$ from the assumption. Assume we have 
\begin{align*}
\sup_{0\leq t \leq T_*}\mathcal{E}^{n}(t) \leq M_0, \quad \mathcal{E}^{n+1}(0)\leq M_0/2.
\end{align*}
Then, from \eqref{Eps_n+1}, we see that 
\begin{align*}
\sup_{0\leq t \leq T_*}\mathcal{E}^{n+1}(t)&\leq \frac{M_0}{2}+CT_*M_0+CT_*\sup_{0\leq t \leq T_*}\mathcal{E}^{n+1}(t)  \cr
&\quad + CT_*\sqrt{M_0}\sqrt{\sup_{0\leq t \leq T_*}\mathcal{E}^{n+1}(t)} +CT_*M_0\sqrt{\sup_{0\leq t \leq T_*}\mathcal{E}^{n+1}(t)}.
\end{align*}
By using Young's inequality, we have
\begin{align*}
(1-3CT_*)\sup_{0\leq t \leq T_*}\mathcal{E}_1^{n+1}(t)&\leq \frac{M_0}{2} +
2 C T_*M_0  +CT_*M_0^2 .
\end{align*}
Therefore, for sufficiently small $T_*$ and $M_0>0$, we can derive
\begin{align*}
\sup_{0\leq t \leq T_*}\mathcal{E}^{n+1}(t)&\leq M_0 .
\end{align*}
This completes the proof.
\end{proof}

\section{Coercivity estimate}
We write the macroscopic part $P(f_1,f_2)$ of the distribution function $(f_1,f_2)$ as 
\begin{align*}
P(f_1,f_2) &= a_1(x,t)\left(\sqrt{\mu_1},0\right)+a_2(x,t)\left(0,\sqrt{\mu_2}\right) +b(x,t)\cdot v \left(m_1\sqrt{\mu_1},m_2\sqrt{\mu_2} \right) \cr
&\quad + c(x,t)|v|^2\left(m_1\sqrt{\mu_1},m_2\sqrt{\mu_2}\right),
\end{align*}
where
\begin{align}\label{abc}
\begin{split}
a_k(x,t) &= \frac{1}{n_{k0}} \int_{\mathbb{R}^3} f_k\sqrt{\mu_k} dv \cr
&\quad- \frac{1}{2n_{10}+2n_{20}}\left(\int_{\mathbb{R}^3} f_1(m_1|v|^2-3)\sqrt{\mu_1} dv+\int_{\mathbb{R}^3} f_2(m_2|v|^2-3)\sqrt{\mu_2} dv\right), \cr
b(x,t) &= \frac{1}{m_1n_{10}+m_2n_{20}}\left(\int_{\mathbb{R}^3} f_1m_1v\sqrt{\mu_1} dv +\int_{\mathbb{R}^3} f_2m_2v\sqrt{\mu_2} dv\right), \cr
c(x,t) &= \frac{1}{6n_{10}+6n_{20}}\left(\int_{\mathbb{R}^3} f_1(m_1|v|^2-3)\sqrt{\mu_1} dv+\int_{\mathbb{R}^3} f_2(m_2|v|^2-3)\sqrt{\mu_2} dv\right),
\end{split}
\end{align}
for $k=1,2$.
We substitute 
\[
(f_1,f_2)=(I-P)(f_1,f_2)+P(f_1,f_2),
\] 
into \eqref{pertff} to get 
\begin{align}\label{split}
\begin{split}
\{\partial_t +v\cdot \nabla_x\} P(f_1,f_2)&=-\{\partial_t +v\cdot \nabla_x-L\} (I-P)(f_1,f_2)  \cr
&\quad+(\Gamma_{11}(f_1)+\Gamma_{12}(f_1,f_2),\Gamma_{22}(f_2)+\Gamma_{21}(f_1,f_2)).
\end{split}
\end{align}
We write L.H.S. of \eqref{split} in the following form:
\begin{multline*}
\bigg\{\left(\partial_ta_1+v\cdot \nabla_x a_1\right)(\sqrt{\mu_1},0)+\left(\partial_ta_2+v\cdot \nabla_x a_2\right)(0,\sqrt{\mu_2}) \cr
+v\cdot\partial_t b(m_1\sqrt{\mu_1},m_2\sqrt{\mu_2}) + \sum_{1\leq i<j \leq 3}v_iv_j(\partial_{x_i}b_j+\partial_{x_j}b_i)(m_1\sqrt{\mu_1},m_2\sqrt{\mu_2}) \cr
+ \sum_{1\leq i \leq 3 }(\partial_{x_i}b_i+\partial_tc)v_i^2(m_1\sqrt{\mu_1},m_2\sqrt{\mu_2})  + |v|^2v\cdot \nabla_x c (m_1\sqrt{\mu_1},m_2\sqrt{\mu_2}) \bigg\},
\end{multline*}
as a linear expansion with respect to the following $17$ basis:
\begin{align}\label{basis17}
\begin{split}
\{(\sqrt{\mu_1},0),(0,\sqrt{\mu_2}),v(\sqrt{\mu_1},0)&,v(0,\sqrt{\mu_2}),\cr
&v_iv_j(m_1\sqrt{\mu_1},m_2\sqrt{\mu_2}),v|v|^2(m_1\sqrt{\mu_1},m_2\sqrt{\mu_2})\}.
\end{split}
\end{align}
Therefore, comparing both sides of \eqref{split}, we obtain the following system:
\begin{align*}
\partial_t a_1 &= l_{a1}+h_{a1}, \cr
\partial_t a_2 &= l_{a2}+h_{a2}, \cr
\partial_{x_i}a_1 + m_1\partial_t b_i &= l_{b1i} + h_{b1i}, \cr
\partial_{x_i}a_2 + m_2\partial_t b_i &= l_{b2i} + h_{b2i}, \cr
\partial_{x_i}b_j+\partial_{x_j}b_i&= l_{bbi}+ h_{bbi}, \quad (i\neq j ) \cr
\partial_{x_i}b_i+\partial_tc &= l_{bci}+ h_{bci}, \cr
\partial_{x_i}c &=  l_{ci} + h_{ci},
\end{align*}
where $(l_{a1},l_{a2},l_{b1i},l_{b2i},l_{bbi},l_{bci},l_{ci})$, and $(h_{a1},h_{a2},h_{b1i},h_{b2i},h_{bbi},h_{bci},h_{ci})$ are the coefficients corresponding to the expansion of $l$ and $h$:
\begin{align*}
&l(f_1,f_2)= -\{\partial_t +v\cdot \nabla_x-L\} (I-P)(f_1,f_2), \cr
&h(f_1,f_2)=(\Gamma_{11}(f_1)+\Gamma_{12}(f_1,f_2),\Gamma_{22}(f_2)+\Gamma_{21}(f_1,f_2)), 
\end{align*}
with respect to \eqref{basis17}. For brevity, we denote
\begin{align*}
\tilde{l} &= l_{a1}+l_{a2}+\sum_{i=1}^3\left(l_{b1i}+l_{b2i}+l_{bbi}+l_{bci}+l_{ci}\right) \cr
\tilde{h} &= h_{a1}+h_{a2}+\sum_{i=1}^3\left(h_{b1i}+h_{b2i}+h_{bbi}+h_{bci}+h_{ci}\right).
\end{align*}

\begin{lemma} We have 
\begin{align*}
\int_{\mathbb{T}^3}a_1(x,t)dx=\int_{\mathbb{T}^3}a_2(x,t)dx=\int_{\mathbb{T}^3}b(x,t)dx=\int_{\mathbb{T}^3}c(x,t)dx=0.
\end{align*}
\end{lemma}
\begin{proof}
This follows from the conservation laws \eqref{conservf} and the definition of $a_1$, $a_2$, $b$, and $c$ in \eqref{abc}.
\end{proof}

\begin{lemma}\emph{\cite{Guo VMB}}\label{abc esti} Let $ 0\leq |\alpha| \leq N $ with $N\geq 3 $, then we have 
\begin{multline*}
\| \partial^{\alpha}a_1 \|_{L^2_x}+\| \partial^{\alpha}a_2 \|_{L^2_x}+\|\partial^{\alpha}b \|_{L^2_x}+\|\partial^{\alpha}c \|_{L^2_x}
\leq \sum_{|\alpha|\leq N-1} \left(\|\partial^{\alpha}\tilde{l} \|_{L^2_x}+\|\partial^{\alpha}\tilde{h} \|_{L^2_x}\right).
\end{multline*}
\end{lemma}
\begin{proof}
The proof can be found in \cite[page 620, Proof of Theorem 3]{Guo VMB}. We omit it.
\end{proof}
\begin{lemma}\label{lh} For sufficiently small energy norm $\mathcal{E}(t)$, we have 
\begin{align*}
&(1) \ \sum_{|\alpha|\leq N-1} \|\partial^{\alpha}\tilde{l}\|_{L^2_{x,v}} \leq C \sum_{|\alpha|\leq N}\|(I-P)\partial^{\alpha}(f_1,f_2)\|_{L^2_{x,v}}, \cr 
&(2) \ \sum_{|\alpha|\leq N} \|\partial^{\alpha}\tilde{h}\|_{L^2_{x,v}} \leq C\sqrt{M} \sum_{|\alpha|\leq N}\|\partial^{\alpha}(f_1,f_2)\|_{L^2_{x,v}}.
\end{align*}
\end{lemma}
\begin{proof}
(1) The proof can be found in \cite[page 616, Lemma 7]{Guo VMB}. We omit it. \newline
(2) Let us define $\{e_i^*\}_{i=1}^{17}$ be the orthonormal basis corresponding to the basis \eqref{basis17}. Then we can write 
\begin{align*}
e_i^* = \sum_{j=1}^{17} C_{ij}e_j, \qquad  h(f_1,f_2) =\sum_{i=1}^{17}  \langle h, e_i^* \rangle_{L^2_v} e_i^*,
\end{align*}
so that
\begin{align*}
\langle h, e_n^* \rangle_{L^2_v} = \sum_{1\leq i,j \leq 17}C_{ij}C_{ni}\langle h, e_i^* \rangle_{L^2_v},
\end{align*}
for $n=1,\cdots,17$. For the estimate of $h$, we compute 
\begin{align*}
\bigg\| \int \partial^{\alpha}h(f_1,f_2) e_i^* dv \bigg\|_{L^2_x} &\leq \bigg\| \int \partial^{\alpha}\Gamma_{11}(f_1) (|v|^k\sqrt{\mu_1}) dv \bigg\|_{L^2_x} + \bigg\| \int \partial^{\alpha}\Gamma_{12}(f_1,f_2) (|v|^k\sqrt{\mu_1}) dv \bigg\|_{L^2_x} \cr
&+ \bigg\| \int \partial^{\alpha}\Gamma_{22}(f_2) (|v|^k\sqrt{\mu_2}) dv \bigg\|_{L^2_x} + \bigg\| \int \partial^{\alpha}\Gamma_{21}(f_1,f_2) (|v|^k\sqrt{\mu_2}) dv \bigg\|_{L^2_x},
\end{align*}
for $k=0,1,2,3$. For sufficiently small $\mathcal{E}(t)$, by Lemma \ref{nonlin esti1}, we have 
\begin{align*}
\bigg\| \int \partial^{\alpha}\Gamma_{mm}(f_m) |v|^k\sqrt{\mu_m} dv \bigg\|_{L^2_x} \leq C\sum_{|\alpha_1|+|\alpha_2|\leq |\alpha|}\bigg\| \|\partial^{\alpha_1}f_m\|_{L^2_v}\|\partial^{\alpha_2}f_m\|_{L^2_v}\bigg\|_{L^2_x}. 
\end{align*}
Similarly, we have from Lemma \ref{nonlin esti}
\begin{align*}
\bigg\| \int \partial^{\alpha}\Gamma_{lm}(f_l,f_m) |v|^k\sqrt{\mu_l} dv \bigg\|_{L^2_x}
\leq C\sum_{|\alpha_1|+|\alpha_2|\leq |\alpha|}\bigg\| \|\partial^{\alpha_1}(f_l,f_m)\|_{L^2_v}\|\partial^{\alpha_2}(f_l,f_m)\|_{L^2_v}\bigg\|_{L^2_x},
\end{align*}
for $l\neq m$. Without loss of generality, we assume that $|\alpha_1|\leq|\alpha_2|$ and apply the Sobolev embedding $H^2 \subset\subset L^{\infty}$ to obtain
\begin{align*}
\sum_{|\alpha|\leq N} \|\partial^{\alpha}\tilde{h}\|_{L^2_{x,v}} &\leq C \sum_{|\alpha_1|\leq |\alpha_2|}\sup_{x\in\mathbb{T}^3}\|\partial^{\alpha_1}(f_1,f_2)\|_{L^2_v} \sum_{|\alpha_2|\leq N} \|\partial^{\alpha_2}(f_1,f_2)\|_{L^2_{x,v}}\cr
&\leq C\sqrt{\mathcal{E}(t)}\sum_{|\alpha|\leq N}\|\partial^{\alpha}(f_1,f_2)\|_{L^2_{x,v}},
\end{align*}
which gives desired result.
\end{proof}
We are now ready to derive the full coercivity estimate. By Lemma \ref{abc esti}, we have 
\begin{align*}
\sum_{|\alpha|\leq N} \| \partial^{\alpha} P(f_1,f_2) \|_{L^2_{x,v}}^2 &\leq  \sum_{|\alpha|\leq N}\left( \| \partial^{\alpha}a_1 \|_{L^2_x}^2+\| \partial^{\alpha}a_2 \|_{L^2_x}^2+\|\partial^{\alpha}b \|_{L^2_x}^2+\|\partial^{\alpha}c \|_{L^2_x}^2\right)  \cr
&\leq \sum_{|\alpha|\leq N-1} \left(\|\partial^{\alpha}\tilde{l} \|_{L^2_x}^2+\|\partial^{\alpha}\tilde{h} \|_{L^2_x}^2\right).
\end{align*}
We then apply Lemma \ref{lh} to get
\begin{align*}
\sum_{|\alpha|\leq N} \| &\partial^{\alpha} P(f_1,f_2) \|_{L^2_{x,v}}^2 \cr
&\leq  C \sum_{|\alpha|\leq N}\|(I-P)\partial^{\alpha}(f_1,f_2)\|_{L^2_{x,v}}^2 + C\sqrt{M} \sum_{|\alpha|\leq N}\|\partial^{\alpha}(f_1,f_2)\|_{L^2_{x,v}}^2.
\end{align*}
Adding $\displaystyle{\sum_{|\alpha|\leq N}\|(I-P)\partial^{\alpha}(f_1,f_2)\|_{L^2_x}^2}$ on each side, we obtain 
\begin{align*}
\sum_{|\alpha|\leq N} \| \partial^{\alpha}(f_1,f_2) \|_{L^2_{x,v}}^2 \leq  \frac{C+1}{1-C\sqrt{M}} \sum_{|\alpha|\leq N}\|(I-P)(\partial^{\alpha}(f_1,f_2))\|_{L^2_{x,v}}^2.
\end{align*}
Combining it with the estimate in Proposition \ref{dissipation}, we derive the following full coercivity estimate 
\begin{align}\label{full coer}
\langle L\partial^{\alpha}(f_1,f_2),\partial^{\alpha}(f_1,f_2)\rangle_{L^2_{x,v}} \leq - \eta \min\left\{(1-\delta),(1-\omega) \right\}\sum_{|\alpha|\leq N} \| \partial^{\alpha}(f_1,f_2) \|_{L^2_{x,v}}^2,
\end{align}
when $\mathcal{E}(t)$ is sufficiently small.

\section{Global existence}
In this section, we extend the local-in-time solution to the global one by establishing a uniform energy estimate. Let $(f_1,f_2)$ be the classical local-in-time solution constructed in Theorem \ref{theo_loc_ex}. 
We take $\partial^{\alpha}$ on \eqref{pertf1} and take inner product with $\partial^{\alpha}f_1$ in $L^2_{x,v}$ to have
\begin{align}\label{f11}
\begin{split}
\frac{1}{2}\frac{d}{dt} \| \partial^{\alpha}f_1 \|_{L^2_{x,v}}^2&=\langle \partial^{\alpha}L_{11}(f_1),\partial^{\alpha}f_1 \rangle_{L^2_{x,v}}+ \langle \partial^{\alpha}L_{12}(f_1,f_2),\partial^{\alpha}f_1 \rangle_{L^2_{x,v}} \cr
&\quad +\langle \partial^{\alpha}f_1,\partial^{\alpha}(\Gamma_{11}+\Gamma_{12}) \rangle_{L^2_{x,v}} .
\end{split}
\end{align}
Similarly, we get from \eqref{pertf2} that
\begin{align}\label{f22}
\begin{split}
\frac{1}{2}\frac{d}{dt} \| \partial^{\alpha}f_2 \|_{L^2_{x,v}}^2&=\langle \partial^{\alpha}L_{22}(f_2),\partial^{\alpha}f_2 \rangle_{L^2_{x,v}}+ \langle \partial^{\alpha}L_{21}(f_1,f_2),\partial^{\alpha}f_2 \rangle_{L^2_{x,v}} \cr
&\quad +\langle \partial^{\alpha}f_2,\partial^{\alpha}(\Gamma_{22}+\Gamma_{21}) \rangle_{L^2_{x,v}} .
\end{split}
\end{align}
Combining \eqref{f11} and \eqref{f22} yields
\begin{align*}
\sum_{k=1,2} \frac{1}{2}\frac{d}{dt} \|\partial^{\alpha} f_k \|_{L^2_{x,v}}^2&\leq \langle L\partial^{\alpha}(f_1,f_2),\partial^{\alpha}(f_1,f_2)\rangle_{L^2_{x,v}} \cr
&\quad+\langle \partial^{\alpha} f_1,\partial^{\alpha} (\Gamma_{11}+ \Gamma_{12}) \rangle_{L^2_{x,v}} +\langle \partial^{\alpha} f_2,\partial^{\alpha}(\Gamma_{22}+\Gamma_{21}) \rangle_{L^2_{x,v}}.
\end{align*}
Then the first term of the R.H.S is controlled by the full coercivity estimate \eqref{full coer}, and the nonlinear terms on the second line are estimated by Lemma \ref{nonlin esti1} and Lemma \ref{nonlin esti}:
\begin{multline*}
\sum_{|\alpha|\leq N}\sum_{k=1,2}\left(\frac{1}{2}\frac{d}{dt} \|\partial^{\alpha} f_k \|_{L^2_{x,v}}^2+\eta \min\left\{(1-\delta),(1-\omega) \right\} \| \partial^{\alpha}f_k \|_{L^2_{x,v}}^2\right) \cr
\leq  C_0\sqrt{\mathcal{E}_{N_1,0}(t)}\sum_{|\alpha|\leq N}\|\partial^{\alpha} (f_1,f_2) \|_{L^2_{x,v}}^2.
\end{multline*}
For $M_0$ satisfying Theorem \ref{theo_loc_ex} and \eqref{full coer}, we define
\begin{align*}
M = \left\{\frac{M_0}{2}, \frac{\eta^2 \min\left\{(1-\delta)^2,(1-\omega)^2 \right\}}{4C_0^2}\right\}, \qquad T= \sup_{t\in\mathbb{R}^+}\{t ~|~ \mathcal{E}_{N_1,0}(t) \leq 2M \} >0.
\end{align*}
We restrict our initial data to satisfy the following energy bound:
\begin{align*}
\mathcal{E}_{N_1,0}(0) \leq M \leq 2M_0.
\end{align*}
Once we define
\begin{align*}
y(t) =\sum_{|\alpha|\leq N}\sum_{k=1,2}\|\partial^{\alpha} f_k \|_{L^2_{x,v}}^2,
\end{align*}
then $y(t)$ satisfies
\begin{align*}
y'(t) + 2\eta \min\left\{(1-\delta),(1-\omega) \right\}y(t) &\leq 2C_0\sqrt{\mathcal{E}_{N_1,0}(t)} y(t)  \cr
&\leq \eta \min\left\{(1-\delta),(1-\omega) \right\}y(t).
\end{align*}
Thus we obtain
\begin{align*}
y(t) \leq e^{- \eta \min\left\{(1-\delta),(1-\omega) \right\}t}y(0) \leq y(0) \leq M  < 2M,
\end{align*}
and which is possible only when $T=\infty$. Note that this also gives 
\begin{align*}
\sum_{|\alpha|\leq N}\|\partial^{\alpha} (f_1(t),f_2(t)) \|_{L^2_{x,v}}^2 \leq e^{-\eta \min\left\{(1-\delta),(1-\omega) \right\}t}\sum_{|\alpha|\leq N}\|\partial^{\alpha} (f_1(0),f_2(0)) \|_{L^2_{x,v}}^2.
\end{align*}
Now we consider the general case of $f$ having momentum derivatives. Taking $\partial^{\alpha}_{\beta}$ on \eqref{pertf1} and \eqref{pertf2} and applying an inner product with $\partial^{\alpha}_{\beta} f_1$ and $\partial^{\alpha}_{\beta} f_2$, respectively, we have 
\begin{align}\label{d1}
\begin{split}
\frac{1}{2}\frac{d}{dt}\|\partial^{\alpha}_{\beta} f_1 \|_{L^2_{x,v}}^2&+(n_{10}+n_{20})\|\partial^{\alpha}_{\beta} f_1 \|_{L^2_{x,v}}^2 = - \sum_{i=1}^3 \langle \partial^{\alpha+\bar{k}_i}_{\beta-k_i}f_1, \partial^{\alpha}_{\beta} f_1 \rangle_{L^2_{x,v}} \cr
	&\quad + (n_{10}+n_{20})
	\langle \partial_{\beta}P_1\partial^{\alpha}f_1,\partial^{\alpha}_{\beta} f_1 \rangle_{L^2_{x,v}}
	+\langle \partial^{\alpha}_{\beta}L_{12}^2(f_1,f_2), \partial^{\alpha}_{\beta} f_1 \rangle_{L^2_{x,v}} \cr
	&\quad +\langle \partial^{\alpha}_{\beta}(\Gamma_{11}(f_1)+\Gamma_{12}(f_1,f_2)),\partial^{\alpha}_{\beta} f_1 \rangle_{L^2_{x,v}},
\end{split}
\end{align}
and
\begin{align}\label{d2}
\begin{split}
\frac{1}{2}\frac{d}{dt}\|\partial^{\alpha}_{\beta} f_2 \|_{L^2_{x,v}}^2&+(n_{10}+n_{20})\|\partial^{\alpha}_{\beta} f_2 \|_{L^2_{x,v}}^2 = - \sum_{i=1}^3 \langle \partial^{\alpha+\bar{k}_i}_{\beta-k_i}f_2, \partial^{\alpha}_{\beta} f_2 \rangle_{L^2_{x,v}} \cr
	&\quad + (n_{10}+n_{20})
	\langle \partial_{\beta}P_2\partial^{\alpha}f_2,\partial^{\alpha}_{\beta} f_2 \rangle_{L^2_{x,v}}
	+\langle \partial^{\alpha}_{\beta}L_{21}^2(f_1,f_2), \partial^{\alpha}_{\beta} f_2 \rangle_{L^2_{x,v}} \cr
	&\quad +\langle \partial^{\alpha}_{\beta}(\Gamma_{22}(f_2)+\Gamma_{21}(f_1,f_2)),\partial^{\alpha}_{\beta} f_2 \rangle_{L^2_{x,v}}.
\end{split}
\end{align}
Combining \eqref{d1} and \eqref{d2}, and applying the H\"{o}lder inequality and Young's inequality, we can obtain 
\begin{multline*}
\sum_{k=1,2}\left(\frac{1}{2}\frac{d}{dt}\|\partial^{\alpha}_{\beta} f_k \|_{L^2_{x,v}}^2 + (n_{10}+n_{20}-2\epsilon)\|\partial^{\alpha}_{\beta} f_k \|_{L^2_{x,v}}^2\right) \cr
\leq \frac{1}{2\epsilon}\sum_{k=1,2}\sum_{i=1}^3\|\partial^{\alpha+\bar{k}_i}_{\beta-k_i}f_k\|_{L^2_{x,v}}^2 +\frac{C}{2\epsilon}\sum_{k=1,2}\| \partial^{\alpha}f_k\|_{L^2_{x,v}}^2  +C \mathcal{E}_{N_1,|\beta|}^{\frac{3}{2}}(t),
\end{multline*}
for some positive constant $\epsilon$ satisfying $(n_{10}+n_{20})/2>\epsilon>0$. We sum this over $|\beta|=m+1$ and multiply both sides with $\epsilon \eta_m$:
\begin{multline*}
\sum_{|\beta|=m+1}\left[\sum_{k=1,2}\left(\frac{\epsilon \eta_m}{2}\frac{d}{dt}\|\partial^{\alpha}_{\beta} f_k \|_{L^2_{x,v}}^2 + \epsilon \eta_m(n_{10}+n_{20}-2\epsilon)\|\partial^{\alpha}_{\beta} f_k \|_{L^2_{x,v}}^2\right)\right] \cr
\leq \sum_{|\beta|=m+1}\left[\frac{\eta_m}{2}\sum_{k=1,2}\sum_{i=1}^3\|\partial^{\alpha+\bar{k}_i}_{\beta-k_i}f_k\|_{L^2_{x,v}}^2 +\frac{C\eta_m}{2}\sum_{k=1,2}\| \partial^{\alpha}f_k\|_{L^2_{x,v}}^2  +C \mathcal{E}_{N_1,|\beta|}^{\frac{3}{2}}(t)\right].
\end{multline*}
Combining the previous cases $|\beta| \leq m$, the R.H.S of the inequality can be bounded by the energy $\mathcal{E}_{N_1,|\beta|}$ with $|\beta|\leq m$ and $\mathcal{E}_{N_1,0}$. Thus, we can conclude from induction that 
\begin{align*}
\sum_{\substack{|\alpha|+|\beta|\leq N \cr |\beta|\leq m+1}}\sum_{k=1,2} \left( C_{m+1}\frac{d}{dt}\| \partial^{\alpha}_{\beta}f_k \|_{L^2_{x,v}}^2+ \eta_{m+1}\| \partial^{\alpha}_{\beta}f_k \|_{L^2_{x,v}}^2\right) \leq C_{m+1}^*\mathcal{E}_{N_1,|\beta|}^{\frac{3}{2}}(t).
\end{align*}
Applying the same continuity argument as to when $\beta=0$, we can construct the global-in-time classical solution. We mention that when $|\beta|=0$, the parameter $\eta_0$ depends on $1-\delta$ and $1-\omega$, and $C_0=1/2$. But when $|\beta|\geq1$, both $C_{m+1}$ and $\eta_{m+1}$ depend on the parameter $\eta_m$. That is why we cannot extract a decay rate depending explicitly on the parameter $\delta$ and $\omega$ when the velocity derivatives are involved. 
For the uniqueness of the solution and $L^2$ stability, we can follow the standard arguments in \cite{Guo whole,Guo VMB,Guo VPB,Yun1}. This completes the proof.\newline\newline

\begin{center}
	{\bf Acknowledgement:}
\end{center}
G.-C. Bae is supported by the National Research Foundation of Korea(NRF) grant funded by the Korea government(MSIT) (No. 2021R1C1C2094843). 
C. Klingenberg is supported in part by the W\"urzburg University performance-orientated fund LOM 2021.
M. Pirner is supported by the Alexander von Humboldt Foundation
S.-B. Yun is supported by Samsung Science and Technology Foundation under Project Number SSTF-BA1801-02.

\end{document}